\documentclass{article}
\usepackage[a4paper,textwidth=6.27in,textheight=9.53in]{geometry}
\usepackage{url}
\usepackage{amssymb,amstext,amsmath,amsthm}
\usepackage{enumitem}
\usepackage{framed}

\usepackage[bookmarksopen=false,breaklinks=true,%
      backref=page,pagebackref=true,plainpages=false,%
      hyperindex=true,pdfstartview=FitH,colorlinks=true,%
      pdfpagelabels=true,linkcolor=blue,%
      citecolor=red,urlcolor=red,hypertexnames=false%
      ]%
   {hyperref}

\newtheorem{theorem}{Theorem}[section]
\newtheorem{corollary}[theorem]{Corollary}
\newtheorem{lemma}[theorem]{Lemma}
\newtheorem{thdef}[theorem]{Theorem and definition}
\newtheorem{proposition}[theorem]{Proposition}
\newtheorem{propdef}[theorem]{Proposition and definition}
\theoremstyle{definition}
\newtheorem{definition}[theorem]{Definition}

\newtheorem{context}[theorem]{Context}
\theoremstyle{remark}
\newtheorem{remark}[theorem]{Remark}

\usepackage{natbib}
\bibpunct{(}{)}{;}{a}{}{,}

\usepackage{tocbibind}

\usepackage[T1]{fontenc}
\usepackage[german,french,english.british]{babel}\frenchsetup{og=«,fg=»}

\def\oge{\leavevmode\raise
.3ex\hbox{\(\scriptscriptstyle\langle\!\langle\,\)}}
\def\feg{\leavevmode\raise
.3ex\hbox{\(\scriptscriptstyle\,\rangle\!\rangle\)}}

\newcommand{\vep}{\alpha}%\newcommand{\vep}{\varepsilon}
\newcommand{\vet}{\beta}%\newcommand{\vet}{\eta}
\renewcommand\leq{\leqslant}
\renewcommand\geq{\geqslant}

\newcommand{\NN}{\mathbb{N}}

\newcommand{\MM}{\mathbb{M}}
\newcommand{\PP}{\mathbb{P}}

\newcommand{\GL}{\mathbb{GL}}
\newcommand{\PGL}{\mathbb{PGL}}

\newcommand{\rh}{\mathrm{h}}

\newcommand{\rL}{\mathrm{L}}
\newcommand{\rN}{\mathrm{N}}
\newcommand{\rR}{\mathrm{R}}

\newcommand{\rZ}{\mathrm{Z}}
\newcommand{\Br}{\mathrm{Br}}
\newcommand{\Char}{\mathrm{char}}
\newcommand{\Id}{\mathrm{Id}}
\newcommand{\End}{\mathrm{End}}

\newcommand{\Int}{\mathrm{Inn}}
\newcommand{\Ker}{\mathrm{Ker}}
\newcommand{\Trd}{\mathrm{Trd}}

\newcommand{\Nrd}{\mathrm{Nrd}}

\newcommand{\Cp}{\mathrm{Cp}}
\newcommand{\Cprd}{\mathrm{Cprd}}
\newcommand{\Cpf}{\mathrm{Cpf}}
\newcommand{\Adeux}{\mathrm{A}_2}

\newcommand\diag[1] {#1_{{\mathrm{\textup{d}}}}}

\newcommand\ttt[1] {{\tt #1}}
\newcommand\cha[1]{#1}

\newcommand \gen[1] {\left\langle{#1}\right\rangle}

\newcommand \sotq[2]{\{\,#1:#2\,\}}

\newcommand \ov {\overline }
\newcommand \op{^{\textrm{\textup{op}}}}
\newcommand \conj[1] {#1^I }

\newcommand \som {\sum\nolimits}

\newcommand \ssi {if, and only if, }

\usepackage{authblk}
\usepackage{doi}
\begin{document}

\title{Constructive \cha{basic} %remarks on the 
theory of central simple algebras}
\author{Thierry Coquand}
\affil{Computer Science and Engineering Department,
University of Gothenburg, Sweden,
\url{coquand@chalmers.se}
}
\author{Henri Lombardi}
\author{Stefan Neuwirth}
\affil{Université de Franche-Comté, CNRS, LmB, FR-25000 Besançon, France,
\url{henri.lombardi@univ-fcomte.fr},
\url{stefan.neuwirth@univ-fcomte.fr}
}
\date{\today}
\maketitle

\tableofcontents

%%%%%%%%%%%%%%%%%%%%%%%%%%%%%%%%%%%%%%%%%%%%%%%%%%%%%%%%%%%%%%%%%%%%
\section*{Introduction}

Wedderburn's theorem (\citealp[Theorem 22]{Wedderburn}; \citealp[Theorem 3.9]{Albert}; \citealp[Theorem VI.2.1]{Curtis}; \citealp[Theorem IX.5.1]{MRR}) states that
any central simple algebra is a matrix algebra over a division algebra. It
is an early instance of a {\em nonconstructive}
result in algebra.\footnote{The standard quaternion algebra~\(A\) over a field~\(K\) is central simple. But if we cannot decide whether \(K\) contains a root of \(X^2+1\), we cannot decide whether \(A\) is a division algebra or  the matrix algebra \(\MM_2(K)\) \citep[Exercise IX.1.4]{MRR}.}
While less famous than Hilbert's basis theorem in this respect, it has also 
played an important role in the development of abstract structures in mathematics, such
as the use of modules satisfying the Noetherian or the Artinian property
in representation theory \citep[\S~VI.2]{Curtis}.
Although nonconstructive, it is also a basic result in the theory of central simple algebras,
and this might explain why the presentation of this theory in a constructive framework \citep[\S~IX.5]{MRR}
stops essentially at giving possible weaker versions of this result that are valid constructively.

The main goal of this paper is to show that using in an essential way ideas
from dynamical algebra \citep{CM,CACM} we can go beyond the results of \citealt{MRR} and give a dynamical version of the splitting theorem that characterises a central simple finite-dimensional algebra~\(A\) over a field~\(F\) as one that becomes a matrix algebra
%when the ground field is extended into a commutative algebra
by a faithfully flat scalar extension
(Theorem~\ref{th3}) and as one such that \(A\otimes A\op\) is a matrix algebra (Theorem~\ref{th2}).
%, as one such that every fundamental system of orthogonal idempotents may be refined into a system of isomorphic ones.
We use this version to establish  constructively most of the basic results of the theory of central simple
algebras. For instance, we show in a constructive setting\footnote{This implies
   that these results are valid in an arbitrary sheaf topos.}
the Skolem-Noether theorem  \citep{Skolem,Noether} (Theorem~\ref{thSkNo}), or the fact that the dimension of a central simple algebra
is a square.

We illustrate these results by giving, in Section~\ref{secBecher}, an elementary constructive
proof\footnote{It uses induction up to \(\omega^{\omega}\) and so is not primitive
  recursive; one should be able to use the work \citealt{CMR} for
  a proof-theoretic analysis of the argument.}
of a theorem by \citet{Becher} (Theorem \ref{thBecher}) which is a consequence of a celebrated
theorem by Merkurjev.\footnote{It seems that the proof of this theorem in the reference \citealt{Merkurjev} is
already essentially constructive, but we leave this for future work.}

This paper is written in Bishop's style of constructive mathematics   \citep{Bi67,MRR,CACM}.
Classical sources we have used are \citet{Albert,Baer,Bla,scharlau,Becher2004,GS,Becher}.

%%%%%%%%%%%%%%%%%%%%%%%%%%%%%%%%%%%%%%%%%%%%%%%%%%%%%%%%%%%%%%%%%%%%
\section{Preliminaries}

\subsection{Matrix decomposition of a ring}% over a not necessarily commutative ring}

%\noindent \ttt{ Albert parle de subalgebra sans que l'élément neutre soit 
%nécessairement le même.
%Sans doute cela crée quelques difficultés quand on lit ses démonstrations.
%Il me semble que Gille et S. utilisent plutôt la terminologie actuelle, dans laquelle une sous-structure doit tout respecter, y compris les constantes.
%En tout cas, j'ai choisi la terminologie moderne comme Gille, dès le début du papier}
%t

When can a ring~\(A\) be represented as a ring of matrices \(\MM_n(R)\)? Such a ring has the natural basis of ``elementary matrices''~\(E_{ij}\) whose only nonzero entry is~\(1_R\) in position~\((i,j)\).

Let \(R\)~be a ring (not necessarily a commutative one). Consider the free right \(R\)-module~\(V=\alpha_1R + \dots + \alpha_nR\) and let \(A=\End_R(V)\).
Then the coordinate projections~\(e_i\) on~\(V\) defined by \(e_i\alpha_j = \alpha_i\delta_{ij}\) form a system of orthogonal idempotents that
sum to~\(1_A\): \(e_ie_j=\delta_{ij}e_j\) and \(e_1+\dots+e_n=1_A\). The maps~\(e_{ij}\) on~\(V\) defined by \(e_{ij} \alpha_k = \alpha_i\delta_{jk}\)
are elements of~\(A\) subject to the equations
\begin{equation}
  \begin{cases}
    e_{ij}e_{jl} = e_{il}\text,\\
    e_{ij}e_{kl} = 0\text{ if \(j\ne k\),}\\
    e_{1}+\dots+e_{q}=1_A\text{ with \(e_i=e_{ii}\).}
  \end{cases}\label{eq:1}
\end{equation}
If we represent elements of~\(A\) by matrices, the~\(e_{ij}\) become the matrices~\(E_{ij}\). %\footnote{Here \(E_{ij}\) is seen as an element of \(\MM_q(R)\). So if \(E'_{ij}\) is the corresponding elementary matrix in \(\MM_q(F)\), then \(E_{ij}=1_{R}\otimes E'_{ij}\).}

Conversely, consider a ring~\(A\).
A sequence of elements~\(e_{ij}\) (\(1\leq i,j\leq n\)) subject to~\eqref{eq:1} is a \emph{matrix decomposition} for~\(A\).\footnote{If  $n\geq 2$, \(A\) is noncommutative.}
We write ~\(A_{ij}=e_iAe_j\). They allow to retrieve a ring~\(R\) as above in two equivalent ways.
\begin{itemize}
\item Let~\(R=A_{11}\). We see that $R$ is a ring with unit $e_1$ and that~\(V=Ae_1\) is a right \(R\)-module. The map \(a\mapsto a_{ij}=e_{1i}ae_{j1}\) vanishes on all~\(A_{kl}\)
  with \((k,l)\ne(i,j)\) and defines a bijection of~\(A_{ij}\) with~\(R\).
  We have \(V = e_{11}R + \dots + e_{n1}R\) with basis \((e_{11},\dots,e_{n1})\).
  Let \(a\in A\); the map \(v\mapsto av\) defines an element of \(\End_R(V)\) represented
  by the matrix \((a_{ij})_{i,j}=\som_{i,j}a_{ij}E_{ij}\), so that  \(A\)~can be seen as an \(R\)-algebra isomorphic to~\(\End_R(V)\) and to~\(\MM_n(R)\).\footnote{The reader can write the corresponding ring morphism from $R$ to $A$.}
\item Recall that the centraliser of a subset~\(B\) of~\(A\) is the subring \(\rZ_A(B)=\sotq{x\in A}{\forall y \in B,\,xy=yx}\).  Let~\(R=\diag{A}:=\rZ_A(\{e_{ij},\,1\leq i,j\leq n\})\): \(R\) consists of those~\(a\in A\) that are diagonal and have constant diagonal with respect to the~\(e_{i}\). Then \(A\) is an \(R\)-algebra.
  The map \(a\mapsto\som_he_{hi}ae_{jh}\) vanishes on all~\(A_{kl}\) with \((k,l)\ne(i,j)\) and defines an $R$-module isomorphism of~\(A_{ij}\) with~\(R\). Moreover, each \(A_{ii}\) is a ring with unit element $e_i$ and the $R$-module isomorphism is also a ring isomorphism.
\end{itemize}

We thus have the following lemma.
%:     Lemma{lemBasic2}
\begin{lemma}[matrix decomposition] \label{lemBasic2}
Assume that a ring~\(A\) has a matrix decomposition \((e_{ij})_{1\leq i,j\leq n}\). 
Let \(A_{ij}=e_iAe_j\) and $\diag{A}=\rZ_A(\{e_{ij},\,1\leq i,j\leq n\})\). Consider \(V=Ae_1\): it is a free right $A_{11}$-module as well as a free right $\diag{A}$-module with basis $(e_{11},\dots,e_{n1})$.
\begin{enumerate}
\item  $A_{11}$ is a ring with unit~\(e_1\), we define  $A_{11}$-module isomorphisms \(A_{ij}\simeq A_{11}\) for each~\(i,j\), and this gives a ring isomorphism \(A\simeq\End_{A_{11}}(V)\simeq \MM_n(A_{11})\).
\item $\diag{A}$ is a subring of $A$ and $A\simeq\End_{\diag{A}}(V)\simeq\MM_n(\diag{A})$, alternatively the subalgebra of \(A\) of elements commuting
  with all \(e_{ij}\).
\end{enumerate}
\end{lemma}
%----------- fin lemma ----------------------------------- 
%
% \begin{proof} Elements of \(A_i\) are characterised by \(e_ixe_i=x\). 
% Let \(\psi_{ij}\colon A_j\to A_i\) be defined by \(\psi_{ij}(x)=e_{ij}xe_{ji}\). 
% Elementary computations show that \(\psi_{ij}\) is an \(F\)-algebra morphism 
% and that \(\psi_{ij}\circ \psi_{ji}=\mathrm{Id}_{A_i}\).\\
% For \(a\in A\) let \(a_{ij}=e_{1i} a e_{j1}\). We have \(a_{ij}=e_1a_{ij}e_1\in A_1\). We note that \(\MM_q(A_1)\) has a natural structure of free \(A_1\)-bimodule of rank \(q^2\)
% and we define \(\varphi\colon A\to \MM_q(A_1)\) by 
% \[\varphi(a)=(a_{ij})=\som_{i,j}a_{ij}E_{ij}=\som_{i,j}E_{ij}a_{ij}.
% \] 
% We claim that \(\varphi\) is an \(F\)-algebra isomorphism. Clearly \(\varphi\) is linear and \(\varphi(e_{k\ell})=E_{k\ell}\). So \(\varphi(1)=\varphi(\sum_k e_k)\) is the identity matrix in \(\MM_q(A_1)\). An elementary computation shows \(\varphi(ab)=\varphi(a)\varphi(b)\). If \(\varphi(a)=0\), we get \(e_iae_j=e_{i1}a_{ij}e_{j1}=0\) for all \(i,j\), hence \(a=\som_{i,j}e_iae_j=0\).
% So \(\varphi\) is an isomorphism.
% \end{proof}
% %
% \end{framed}\footnotetext{Here \(E_{ij}\) is seen as an element of \(\MM_q(A_1)\). So if \(E'_{ij}\) is the corresponding elementary matrix in \(\MM_q(F)\), then \(E_{ij}=E'_{ij}\otimes 1_{A_1}\).}

\subsection{Basic definitions}

We first explain what we mean by a central simple algebra and provide some definitions.

\begin{itemize}
%
%\item  
%%
\item We say that a vector space~\(V\) over a field is \emph{strictly finite} when we know a finite basis of~\(V\).
\item By \emph{algebra} we henceforth always mean an associative unital nonzero algebra~\(A\) over a field~\(F\). Such an \(F\)-algebra is given with a ring morphism \(F\to A\). 
The field~\(F\) is supposed to be discrete: every element is 0 or invertible. So the morphism \(F\to A\) is injective. We identify \(F\) with its image \(F1_A\subseteq A\). Moreover, the algebra~\(A\) will always be a strictly finite vector space over~\(F\), with a basis \((\vep_1,\dots,\vep_m)\),
and we shall always take \(\vep_1\) to be the unit \(1_A\) of the algebra.
The algebra is given by a {\em multiplication table} \[\vep_i\vep_j = \sum_{p=1}^m c^p_{ij} \vep_p\text,\] where the coefficients~\(c^p_{ij}\in F\) are called the \emph{structure constants} of~\(A\) with respect to the basis \((\vep_1,\dots,\vep_m)\).
\item  By \emph{left ideal} of~\(A\), we mean a sub-\(F\)-vector space of~\(A\) with a given basis which is a left \(A\)-module.\footnote{In other words, we consider usual left ideals with an explicit \(F\)-basis, possibly \(\emptyset\).} A left ideal is  \emph{proper} when it is \(\neq A\).
An element \(a\in A\) generates a left ideal \(Aa=\sotq{xa}{x\in A}\), a right ideal \(aA=\sotq{ax}{x\in A}\), and an ideal \(\gen{a}=AaA=\sotq{\som_{i=0}^n x_ia y_i}{n\in\NN, \,x_i,y_i\in A}\).
\item If \(B\) is an \(F\)-algebra, the \emph{commutative direct sum} \(B\otimes_F A\) is constructed as the free left \(B\)-module generated by \(\vep_1,\dots,\vep_m\) with the multiplication table as above.\footnote{When~\(A\) and~\(B\) are \(F\)-algebras, a commutative direct sum of~\(A\) and~\(B\) is a solution of the following universal problem: to find an \(F\)-algebra~\(C\) and two morphisms \(u\colon A\to C\) and \(v\colon B\to C\) such that \(u(a)v(b)=v(b)u(a)\) for all \(a,b\); moreover,
when we have two morphisms \(\varphi\colon A\to D\) and \(\psi\colon B\to D\) such that \(\varphi(a)\psi(b)=\psi(b)\varphi(a)\) for all \(a,b\), then there is a unique morphism \(\theta\colon C\to D\) such that \(\theta u=\varphi\) and  \(\theta v=\psi\).
A solution of this universal problem is \(C=A\otimes_F B\)  with the embeddings \(u\colon A\to C,\, a\mapsto a\otimes 1_B\) and 
 \(v\colon B\to C,\,b\mapsto 1_A\otimes b\), and the multiplication law in~\(C\) is defined by \((a\otimes b)\cdot(a'\otimes b')=(aa')\otimes (bb')\).} We shall also use the notation~\(A_B\).

In consequence, if  \((\vet_1,\dots,\vet_n)\) is an \(F\)-basis of~\(B\), the formal products \(\vet_k\vep_i\) form a basis of \(B\otimes_F A\), with \(\vep_i  \vet_k=\vet_k\vep_i\), and by \(F\)-bilinearity we have
\[
\hbox{if } \vep_i\vep_j = \som_{p=1}^m c^p_{ij} \vep_p \hbox{ and }
 \vet_k\vet_l = \som_{q=1}^n d^q_{kl} \vet_q ,\hbox{ then }  (\vet_k\vep_i)(\vet_l\vep_j) = \som_{p=1}^m  \som_{q=1}^n c^p_{ij} d^q_{kl} (\vet_q\vep_p).
 \]
If \(A=\MM_q(F)\), then \(B\otimes_F\MM_q(F)\simeq \MM_q(B)\).\\
As an \(F\)-vector space, \(B\otimes_F A\) is the tensor product of~\(B\) with~\(A\); we shall use the usual terminology stating that the algebra \(B\otimes_FA\) is the tensor product of the algebras~\(B\) and~\(A\).
\item  By  \emph{subalgebra} of an algebra~\(A\) we intend an \(F\)-algebra~\(B\)
with an injective morphism \(B\to A\);\footnote{In other words, we consider usual subalgebras which are strictly finite over~\(F\).} if we identify~\(B\) with its image in~\(A\),  we have \(1_B=1_A\).  The \(F\)-basis of~\(A\) can be modified so that the \(F\)-basis of~\(B\) is an initial subsequence of the \(F\)-basis of~\(A\). The algebra~\(A\) has a natural structure of left \(B\)-module. When \(A\) is a free left \(B\)-module with basis \((c_1,\dots,c_r)\), we write \(r=\dim_B(A)=[A:B]\) and we have \([A:F]=[A:B][B:F]\).
\item  If \(B\) is a subalgebra of an algebra~\(A\), the centraliser~\(\rZ_A(B)\) is a subalgebra: by the constructive theory of vector spaces over a discrete field, we can compute a basis of this centraliser. The subalgebra \(\rZ_A(A)=\rZ(A)\) is called the \emph{center} of~\(A\). 
\item  That \(A\) is {\em central} means that \(\rZ(A)=F1_A\): if \(xy = yx\) for all~\(y\) then \(x = r1_A\)
for some~\(r\) in~\(F\). In other words, %given the constructive theory of vector spaces over a discrete field, 
if \(x\notin F1_A\) then we can find~\(y\) such that \(xy\neq yx\) (because we can compute the kernel of the linear map \(y\mapsto xy-yx\)). 
\item  That \(A\) is {\em simple} can be expressed by stating that if \(a\neq 0\) then \(AaA=A\); in other words we can find
\(x_i\)'s and \(y_i\)'s such that \(\sum_ix_iay_i = 1_A\). Note that a simple commutative \(F\)-algebra is a finite field extension of~\(F\).
\item We use the following notation: \(\End_F(A)\) is the algebra of endomorphisms of the vector space~\(A\); \(\End_{A/F}(A)\) is the monoid of endomorphisms of the \(F\)-algebra~\(A\). The vector space \(\End_F(A)\) is isomorphic to \(\MM_m(F)\) through the basis \((\vep_1,\dots,\vep_m)\).
\item 
We denote by \(\rL_a^{\!A}\) or \(\rL_a\) the \(F\)-linear endomorphism of~\(A\) defined by \(\rL_a x = ax\). Similarly \(\rR_b^Ax =\rR_bx = xb\).
These maps embed~\(A\) and \(A\op\) into \(\End_F(A)\). 
We have \(\rL_a\rR_b = \rR_b\rL_a\).
This provides a natural map \(A\otimes_F A\op\rightarrow\End_F(A)\simeq \MM_m(F)\). When the \(F\)-basis of~\(A\) is clear from the context we say that we consider the canonical map \(A\otimes_F A\op\rightarrow \MM_m(F)\).  
\item  Note that if \(\varphi\in\End_F(A)\) satisfies \(\varphi \rL_a = \rL_a \varphi\) for all \(a\in A\),  then \(\varphi\) is \(\rR_b\) with \(b = \varphi(1)\). 
So, with the embeddings of the previous item, \(A\op\) is the centraliser of~\(A\) in \(\End_F(A)\) and  
\(A\) is the centraliser of \(A\op\) in \(\End_F(A)\).
\item  \(A\) is \emph{split by} a commutative \(F\)-algebra~\(K\) if the \emph{scalar extension} \(A_K=K\otimes_F A\) is isomorphic to a matrix algebra \(\MM_q(K)\) (recall that \(K\) is supposed to be nonzero); we say that \(K\) \emph{splits}~\(A\). This definition allows us to develop the basic theory of central simple algebras constructively. It is equivalent in classical mathematics to the classical one, i.e.\ \(K\) can be replaced by a finite field extension: in fact, if we know  a field quotient \(K/J=L\) of~\(K\), we have 
  \(A_L=L\otimes_F A\simeq L\otimes_K A_K\simeq L\otimes_K\MM_q(K)\simeq \MM_q(L)\).
  If the ground field~\(F\) itself splits~\(A\), i.e.\ if \(A\)~is isomorphic to an~\(\MM_q(F)\), we simply say that \(A\) is \emph{split}. 
\end{itemize}

\subsection{Basic examples}
\label{sec:examples}

\begin{itemize}
\item A matrix algebra \(\MM_q(F)\) is central simple. 
\item If \(B\) is central simple, then \(\MM_q(B)\) is central simple.  
\item  An \(F\)-algebra~\(D\) which is a division ring is simple. Its center~\(K\) is a finite field extension of~\(F\) and the \(K\)-algebra~\(D\) is called a \emph{division algebra}.
\item Recall also that \(\MM_q(\MM_r(B))\simeq \MM_{qr}(B)\).   
\end{itemize} 

The founding example of a central simple algebra that is not a matrix algebra is the quaternion algebra introduced by William Rowan Hamilton in 1843: it is
generated by two elements \(x\) and \(y\)
satisfying \(x^2 =  y^2 = -1\) and \(yx = -xy\). \citealt{Dickson14} generalised this example to the notion of \emph{cyclic algebra}.

Let \(F\) be a discrete field, \(P\) in \(F[X]\) monic irreducible of degree \(n > 0\), and \(L = F[X]/P=F[x]\), where \(P(x) = 0\),
with basis \((1,x,\dots,x^{n-1})\) over \(F\). We assume that \(L\) is a cyclic extension of \(F\): we have an \(F\)-automorphism \(\sigma\)
of \(L\) such that \(\sigma^n(x) = x\) and \(\sigma^i(x) \neq x\) if \(0<i<n\).
We consider the algebra \(A\) of dimension \(n^2\) generated by the formal products \(x^iy^j,~i<n,~j<n\) subject to
\(yxy^{-1} = \sigma(x)\) and \(y^n = a\), where \(a\) is a given element in \(F^{\times}\).

 The quaternion algebra, over e.g.\ the field of rationals, is the special case where \(n=2\), \(P(X)=X^2+1\), and \(a = -1\).

\begin{proposition}
  The \(F\)-algebra \(A\) is central and simple.
\end{proposition}

\begin{proof}
  An element of \(A\) is of the form \(u = l_0+l_1y+\dots+l_{n-1}y^{n-1}\) and we have \(x u x^{-1} = l_0+ l_1 (x/\sigma(x)) y+ \dots+l_{n-1} (x/\sigma^{n-1}(x)) y^{n-1}\).
So, if \(u\) is in the center of~\(A\), we should have \(l_i = 0\) for \(0<i<n\), so \(u\) is in \(L\). But then \(yuy^{-1} = \sigma(u)\) and \(u\) is in \(F\).
This shows that \(A\) is {\em central}.

In order to show that \(A\) is simple, let \(I\) be a left and right ideal and let \(u\) be an element in \(I\) with \(u\neq 0\).
We write as before \(u = l_0+l_1y+\dots+l_{n-1}y^{n-1}\). Let \(N(u)\) be  the number of
coefficients \(l_i\) of \(u\) that are \(\neq 0\). We show that \(1_A\) is in \(I\) by induction on \(N(u)\).

If we have \(l_{p} \neq 0\) and \(l_i = 0\) for \(i<p\), then \(v = (l_py^p)^{-1}u\) is also in \(I\)
and is of the form
\(v = 1_A + \sum_{i>0}m_i y^i\). If some \(m_i\) are \(\neq 0\), we consider \(xvx^{-1} = 1_A + \sum_{i>0} m_i (x/\sigma^i(x)) y^i\)
which is also in \(I\). Then  \(u' = v - xvx^{-1}\) is an element of~\(I\) and we have \(N(u')<N(u)\).
\end{proof}

\subsection{Some general results about algebras}
\label{sec:some-general-results}

Recall that all \(F\)-algebras are supposed to be finite-dimensional vector spaces.
%:     Lemma{lemAlg4}
\begin{lemma} \label{lemAlg4} Let \(A\) be an \(F\)-algebra. An endomorphism of the vector space~\(A\) is injective \ssi it is surjective. An element~\(a\in A\)
is left-regular \ssi it is right-invertible \ssi it is invertible.
\end{lemma}

\begin{proof}
The endomorphism \(x\mapsto ax\) of the finite-dimensional vector space~\(A\) is injective \ssi it is surjective \ssi the endomorphism \(x\mapsto xa\) is surjective.
\end{proof}  

An invertible element~\(u\) of an \(F\)-algebra~\(A\) defines an \emph{inner \(F\)-automorphism of}~\(A\),
denoted by \(\Int(u)\) and defined by \(\Int(u)(a)=uau^{-1}\) for all \(a\in A\).

%:prop5 autom of GLq F
\begin{proposition}\label{prop5}
Let \(\sigma\) be an \(F\)-automorphism of \(\MM_q(F)\). There is an~\(u\) in \(\GL_q(F)\) such that \(\sigma=\Int(u)\).
%  \(\sigma(a) = uau^{-1}\) for~\(a\) in \(\MM_q(F)\). 
  \end{proposition}  
In fact, \(u\) is well-defined up to a unit of~\(F\) since \(uau^{-1}=vav^{-1}\) for all~\(a\) means that \(v^{-1}u\) is in the center \(F1_{\MM_q(F)}\) of \(\MM_q(F)\).
So the group of automorphisms of \(\MM_q(F)\) is isomorphic to the projective linear group~\(\PGL_q(F)\), i.e.\ the quotient of the general linear group~\(\GL_q(F)\) by the nonzero scalar matrices.

% \hum{notations \(\GL_q(F)\) et \(\PGL_q(F)\) non introduites.}

\begin{proof}
  Let \(e_{ij}=\sigma(E_{ij})\), \(1\leq i,j\leq q\): then the~\(e_{ij}\) form a matrix decomposition for~\(\MM_q(F)\), so that the \(e_{ii}=e_i\) form a system of nonzero commuting orthogonal idempotents with \(e_{1}+\dots+e_{q}=1\). Consider the nonzero vector spaces~\(V_i=e_i(F^q)\): we have \(v\in V_i\) \ssi \(e_iv=v\); the map \(v\mapsto e_{ji}v\) vanishes on~\(V_k\), \(k\ne i\), and has range~\(V_j\): it defines an isomorphism of~\(V_i\) onto~\(V_j\), so that the~\(V_i\) all have the same dimension and decompose~\(F^q\) into their direct sum. Now take a nonzero \(v_1\in V_1\) and let \(v_j=e_{j1}v_1\in V_j\); note that \(e_{1j}v_j=e_{1j}e_{j1}v_1=e_{1}v_1=v_1\), so that~\(v_j\ne0\). Then \(e_{ij}v_{j}=e_{i1}e_{1j}v_j=e_{i1}v_1=v_i\) and \(e_{ij}v_k=0\) if \(k\ne j\): take the inverse~\(u\) of the matrix whose columns are the \(v_j\)'s.
 %  We follow the elementary proof in \citealt{Semrl}. Let \(F^q\) denote the vector space of column vectors. Let \(y_0,y,z\)  in \(F^q\)  such that
 %   \(\sigma(y_0y^T)z\neq 0\). We define \(u\in \End_F(F^q)=\MM_q(F)\) by 
 %   \(ux = \sigma(xy^T)z\) for  \(x\in F^q\). We have for each~\(a\) in \(\MM_q(F)\)
 %  \[uax = \sigma(axy^T)z = \sigma(a)\sigma(xy^T)z = \sigma(a)ux \quad \hbox{for all } x\in F^q\]
 %  and so \(ua = \sigma(a)u\). Furthermore \(uy_0 = \sigma(y_0y^T)z\neq 0\). Hence for an arbitrary \(t\in F^q\) we can find an element \(c\in \GL_q(F)\) such that \(cuy_0=t\). Letting \(b=\sigma^{-1}(c)\) we get 
 % \(uby_0 =\sigma(b)uy_0 =cuy_0= t\) 
 %  and hence \(u\) is surjective. So  \(u\in\GL_q(F)\) and \(\sigma(a) = uau^{-1}\)
 %  for all \(a\in\MM_q(F)\).
\end{proof}

%r
%:     Remark{remAutGLnK}
\begin{remark} \label{remAutGLnK} 
  This proof is different from the original proof by \citet[Satz~40]{Skolem}. There is yet another proof by \citet{Semrl}.
  The automorphism group of \(\GL_q(F)\) and other ``classical groups'' for a (skew) field~\(K\) is a classical topic \citep[see e.g.][]{Dd71}.  
\end{remark}
%----------- fin remark ---------------------------------- 

The same argument works for an arbitrary (commutative) ring~\(R\) instead of the field~\(F\) whose ``Picard group''~\(\operatorname{Pic}(R)=0\), i.e.\ such that the projective \(R\)-modules of finite type and of rank~\(1\) are free: then we can consider an \(R\)-basis for \(V_1\) and proceed with their images by~\(e_{j1}\). This yields
\begin{proposition}\label{prop5anneau}
Let \(\sigma\) be an automorphism of the algebra of matrices~\(\MM_q(R)\) with coefficients in a ring~\(R\) such that \(\operatorname{Pic}(R)=0\). There is an~\(u\) in \(\GL_q(R)\) such that \(\sigma=\Int(u)\).
%  \(\sigma(a) = uau^{-1}\) for~\(a\) in \(\MM_q(F)\). 
  \end{proposition}

%%%%%%%%%%%%%%%%%%%%%%%%%%%%%%%%%%%%%%%%%%%%%%%%%%%%%%%%%%%%%%%%%%%%
\section{Central simple algebras: basic results}

  \citet[Theorem 22]{Wedderburn} proved that any simple algebra is isomorphic to  \(\MM_q(D)\)
  for a suitable division algebra~\(D\).
  The proof uses the law of excluded middle and we shall prove constructively a
  dynamical form\footnote{In \citealt{Gonthier}, the problem with nonconstructiveness  is solved by using a double-negation translation. It would be interesting to connect  these two a priori different methods, but we leave this for future work.}
  of this result (Theorem~\ref{th1}).

\medskip An important consequence of Wedderburn's theorem is that an algebra~\(A\) is central simple \ssi A becomes a matrix algebra after a suitable scalar extension, \ssi \(A\otimes _F A\op\) is a matrix algebra (Theorem~\ref{th2} and Proposition \ref{prop4}). Also, the important corollaries~\ref{corth2} and \ref{cor2th2} seem to have no simple elementary proof. 

%% \medskip
%% The equivalence between being central simple and splittable 
%% by a finite field extension 
%% is nontrivial and one of the first basic result of
%% the classical theory. One can refine this by stating that \(A\)  is splittable by a  {\em separable} finite field extension.

%% Constructively, the statement has to be interpreted dynamically: this extension will be replaced with a commutative \(F\)-algebra (Theorems \ref{th3} and \ref{th4}).

%% A possible interpretation of what we do is that we prove that~\(A\) becomes a matrix algebra in the topos modelling the algebraic closure (or the separable closure) of~\(F\).

\subsection{Wedderburn's theorem}

\begin{proposition}\label{prop2}
If \(A\) is simple and has a nonzero nonregular element \(a\), then \(Aa\) contains a nontrivial idempotent.
\end{proposition}

\begin{proof}
  The hypothesis means simply that we have a nonzero \(a\in A\) such that \(Aa\neq A\). The conclusion will be that \(Aa\) contains a nonzero idempotent. 

There is a \(b\in A\) such that \(aba\ne0\): as \(AaA=A\), we can find a finite number of elements \(x_i\)'s and \(y_i\)'s in~\(A\) such that \(\sum x_iay_i=1_A\); squaring this yields \(\sum x_iay_ix_jay_j\ne0\), so that at least one of the \(ay_ix_ja\)'s is nonzero. If \(Aaba\neq Aa\), we replace~\(a\) by~\(aba\) with~\(Aaba\) strictly contained in~\(Aa\) and we are done by induction on the dimension of~\(Aa\). So we may assume that \(Aaba= Aa\) and \(a=xaba\) for some~\(x\in A\), so that \(u:=ba=bxau\). Let \(e=bxa\): \(e\ne0\) and \(e\in Aa\), so that \(e\ne1_A\); the map \(c\mapsto cu\) on the finite-dimensional vector space~\(Aa\) is surjective and therefore injective, and as \(e^2u=ebxau=eu\), we have \(e^2=e\). 
%   Lemma \ref{lemAlg4} gives an idempotent \(e\in Aa\) and it remains to consider the case where \(e=0\) and \(a\) is nilpotent.
% We can assume that \(a\neq 0\) and \(a^2=0\). 
% Since \(A\) is simple, we write \(\sum_iy_iaz_i = 1\), so 
% \[\Big(\som_iy_iaz_i\Big)^2= \som_{ik}y_iaz_iy_kaz_k = \som_{ik}y_iav_{ik}az_k = 1.\] 
% So there is a \(b=v_{ik}\) such that \(a_1=aba\neq 0\). Note that \(a_1^2=0\).
% If \(Aa_1\neq Aa\), we replace~\(a\) by \(a_1\) with \(Aa_1\) strictly contained in \(Aa\) and we are done by induction on the dimension of~\(Aa\). So we may assume that \(Aa_1= Aa\). We have an \(x\in A\) such that \(xaba=a\). Let \(a_2=ba\) and \(f=bxa\): \(a_2\neq 0\) since \(aa_2=a_1\neq 0\),  \(fa_2=a_2\),  \(f\neq 1\) %is not invertible 
% since \(f\in Aa\), and \(f\neq 0\) since  \(fa_2=a_2\neq 0\). We have \((f^2-f)a_2=0\). If \(f^2-f=0\) we are done.
% If \(y=f^2-f\neq 0\) then \(Ay\) is a nonzero sub-left ideal of \(Aa\) and it is proper since \(Aya_2=0\) and \(Aaa_2\neq 0\). So we can replace~\(a\) with \(f^2-f\) and we are done by induction on the dimension of~\(Aa\). 
 \end{proof}  

The next result (Theorem \ref{th1}) is  a constructive/dynamical way to formulate the classical fact
that if~\(A\) is simple then \(A\) can be written 
\(B\otimes_F\MM_q(F)\simeq \MM_q(B)\), where \(B\) is a division
algebra (i.e.\ the classical formulation of Wedderburn's theorem).
Let us start with an essential remark. 

%l
%:     Lemma{lemBasic1}
\begin{lemma} \label{lemBasic1} Let \(A\) be a simple \(F\)-algebra.
If \(e\) is a nontrivial idempotent of~\(A\), then \(eAe\) is again a simple \(F\)-algebra with~\(e\) as unit element (note that it is not a subalgebra of~\(A\)).
\end{lemma}
%----------- fin lemma ----------------------------------- 
%
\begin{proof}
  If we let \(A'=eAe=\sotq{x\in A}{x=exe}\), then \(A'\) is stable by addition and multiplication, and \(e\) is a unit element in \(A'\).
Moreover if \(x\neq 0\) in \(A'\) we write \(\sum_iy_ixz_i = 1\) in~\(A\), which gives
\(\sum_i(ey_ie)x(ez_ie)= \sum_ie(y_i xz_i)e= e\) in \(A'\). 
\end{proof}

Furthermore \(eAe\subsetneq A\) if \(e\neq 1\). 
Note that letting \(f=1-e\) we get a direct sum of  \(F\)-subspaces  \(A=eAe\oplus eAf\oplus fAe\oplus fAf\).
By Proposition \ref{prop2}, if we find a noninvertible element in \(eAe\) then we can split~\(e\) into smaller idempotents.

\smallskip

\begin{theorem}[Wedderburn's theorem, constructive form]\label{th1}
  Let \(A\) be a simple \(F\)-algebra. If \(A\) contains a nonzero noninvertible element, then \(A\) can be written as \(\MM_q(B)\),
  where \(q>1\) and \(B\) is a simple \(F\)-algebra.   Moreover if~\(A\) is central simple, then so is~\(B\).
\end{theorem}

%\hum{Attention, je vois « Wedderburn's theorem, first constructive form », puis « first splitting theorem » et « second splitting theorem » sans mention de Wedderburn. Je propose la terminologie suivante : « Wedderburn's theorem, constructive form »,  « Wedderburn's theorem, dynamical form », first constructive « Wedderburn's theorem, separable dynamical form ».}

\noindent N.B.: 
Note that \(q^2[B:F]=[A:F]\).

\begin{proof} By Lemma \ref{lemBasic2}, it is enough to prove that \(A\) admits a
matrix decomposition. By Proposition~\ref{prop2}, \(A\) contains a nontrivial idempotent~\(e\).  We start with the orthogonal nontrivial idempotents \(e, f\) where \(e+f=1_A\)  and we shall refine this decomposition \((e,f)\) of \(1_A\) into a finer one
\((e_1,\dots,e_n)\) as long as we don't get a matrix decomposition. This terminates because \(A\)~is a strictly finite vector space.

We define \(A_{ij} = e_iAe_j\). We have \(A_{ij}A_{jl} = e_iAe_je_jAe_l = e_iAe_l=A_{il}\)
(\(Ae_jA=A\) since \(A\) is simple) and \(A_{ij}A_{kl} = 0\) if \(j\neq k\).
%For each~\(a\) in \(A_{ij}\) we have \(e_i a = a e_j = a\)  and \(e_l a = ae_k = 0\) if \(l\neq i\) and \(j\neq k\).

For each \(i>1\) we have \(A_{1i}A_{i1}= A_{11}\). So we can find \(a_{1i}\) in \(A_{1i}\) and \(a_{i1}\in A_{i1}\) such that \(a_{1i}a_{i1} = a_1 \neq 0\) in \(A_{11}\). 
If \(A_{11}a_1\neq A_{11}\), Proposition \ref{prop2} gives a nontrivial idempotent in \(A_{11}\) which refines our decomposition. So we can assume
\(A_{11}a_1= A_{11}\),  and \(a_1\) invertible in \(A_{11}\) by Lemma \ref{lemAlg4}.
Let  \(b_1\in A_{11}\) with \(b_1a_1=a_1b_1=e_1\); let \(e_{1i}=b_1a_{1i}\in A_{1i}\) and \(e_{i1}=a_{i1}\); we have \(e_{1i}e_{i1} = e_{1}\). This is done for all \(i>1\).
\\ 
We define then \(e_{11}=e_1\) and \(e_{ij} = e_{i1}e_{1j}\) for \(i,j> 1\);  we have \(e_{ij}e_{jl} = e_{il}\) for all \(i,j,l\)
  and \(e_{ij}e_{kl} = 0\) if \(k\neq j\).
Note that \(e_{ii}= e_{i1}e_{1i}\) is nonzero since \(e_{ii}e_{i1}=e_{i1}\neq 0\). So \(e_{ii}\) is a nonzero idempotent in \(A_{ii}\). If \(e_{ii}\neq e_i\) we can split
  further \(e_i\). So we can assume \(e_i = e_{ii}\) and we have a matrix decomposition for~\(A\).
  The \(F\)-algebra \(A_{11}\) is simple by Lemma~\ref{lemBasic1}. % Let \(x\neq 0\) in \(A_{11}\); we have  \(Y_\ell, Z_\ell\in  A\) such that \(\som_{\ell} Y_\ell xZ_\ell= 1_{A}\). So considering the
  % coefficient in position \((1,1)\) with respect to to the matrix decomposition \(\MM_q(A_{11})\simeq A\), we get an equality \(\som_{r} y_{r}xz_{r}=1_{A_{11}}\).
  
  Finally, if \(x\in A_{11}\) commutes with all elements of \(A_{11}\), then \(x 1_{\MM_q(A_{11})}\) commutes with all matrices in \(\MM_q(A_{11})\), so \(A_{11}\) is central if \(A\) is central. %\footnote{We can also use Lemma \ref{lemAlg3+}.}
  And \(A_{11}\) is the central simple \(F\)-algebra~\(B\) we are looking for.
\end{proof}

\noindent N.B.: Dynamically we do as if \(e_1Ae_1\) is a division algebra, and as if we cannot split further  any   \(e_i\) for \(i>1\):
we do this along the lines of \citealt[Theorem 3.9]{Albert}.

\medskip In classical mathematics, we deduce from Theorem \ref{th1} that we can force~\(B\) to be a division algebra since otherwise the theorem works for~\(B\), which can be written in the form \(\MM_r(C)\) with \(r>1\) and we are done by induction. So if we have a test, for simple \(F\)-algebras, telling us whether  there is a nonzero noninvertible element, we get the classical form of Wedderburn's theorem.

%%%%%%%%%%%%%%%%%%%%%%%%%%%%%%%%%%%%%%%%%%%%%%%%%%%%%%%%%%%%%%%%%%%%
\section{Central simple algebras:  dynamical methods} 

%\cha{The dynamical method \hum{``for~\(K\) (zerodivisors) or for~\(F\) (field extensions)''~? (En fait les deux!) }\citep[Chapter~7]{CM,CACM}
%  allows us to work with a commutative algebra~\(K\) as if it was a discrete field.\footnote{Recall that by~\(F\)-algebra we always intend an algebra given with a finite basis as an \(F\)-vector space.} If, in doing so, we find an obstacle in a proof, this means that we have found a nonzero noninvertible element~\(a\) in~\(K\). In this case we replace~\(K\) by \(K/\!\gen{a}\), which is a nonzero quotient that is still a commutative algebra. Moreover, the computations in~\(K\) that have been done before the obstacle remain valid in \(K/\!\gen{a}\).}

The work \citealt{CM} interprets in a topos-theoretic way the dynamical interpretation of the algebraic closure of a field
as presented in \citealt{D5}: one proves intuitionistically results about algebraically closed fields, and one obtains a corresponding
statement involving finite extensions of a discrete fields by adding successively formal roots of monic polynomials.

\subsection{Splitting central simple algebras}
   
\begin{lemma}\label{algclos}
  If \(A\) is a central simple \(K\)-algebra and \(K\) is algebraically closed, then \(A\)  is split.
\end{lemma}

\begin{proof}
\cha  {Our proof is by induction on the \(K\)-dimension~\(m\) of~\(A\), the case \(m=1\) being trivial. If \(m>1\), the minimal polynomial \(f(X)\in F[X]\) of \(\vep_2\) may be written \((X-a_1)\cdots(X-a_r)\) with \(r>1\) and the \(a_i\)'s in~\(F\). So \((\vep_2-a_1\vep_1)\cdots(\vep_2-a_r\vep_1)=0\). Since the factors are nonzero, \(\vep_2-a_1\vep_1\) is a nonzero noninvertible element of~\(A\). By Wedderburn's theorem \ref{th1},  \(A\simeq \MM_q(B)\) for some \(q>1\). We are done by induction.}
%If \(m>1\), the minimal polynomial \(f(X_1)\in F[X_1]\) of \(\vep_2\) is of degree \(r>1\). We consider \(K=F[\alpha_2]\) where \(\alpha_2\) is a root of~\(f\) in~\(L\).
%%\(=F[X_1]/(h)\) where \(h\) is an irreducible factor of~\(f\) (\(\alpha_2\) is the image of \(X_1\) in~\(K\)). 
%We get \(f(T)=(T-\alpha_2)g(T)\) in \(K[T]\). So \((\vep_2-\alpha_2 \vep_1)g(\vep_2)=0\) in \(K[\vep_2]=K\otimes_FF[\vep_2] \subseteq A_K\). 
%%Corollary \ref{corth2} says that \(A_K\) is central simple over~\(K\). 
% So the element \(\vep_2-\alpha_2 \vep_1\) is  nonzero noninvertible in~\(A\).
% Wedderburn's theorem \ref{th1} says that  we have \(A\simeq \MM_r(B)\) with some \(r>1\) and~\(B\) a central simple \(F\)-algebra.
% By induction, \(B\) is split and so is~\(A\).
\end{proof}

%\hum{Enlever le commentaire suivant:} Using Corollary \ref{corth2}, we deduce that if \(F\) is embedded in an algebraically closed field~\(K\) and \(A\) is a central simple algebra over~\(F\) then \(K\) splits~\(A\).

%Since the reasoning of the previous sections are all valid intuitionistically, we can interpret this result 
%dynamically, i.e.\ as if the minimal polynomial of~\(\varepsilon_2\) was split.\footnote{Another equivalent formulation is that it is always possible to build
%  a sheaf topos in which \(F\) is embedded in an algebraically closed field \citep{CM}.}
 Following \citealt{CM}, we obtain the following result.

 \begin{sloppypar}
   \begin{thdef}[first splitting theorem]\label{th3}\label{definotadeg} Let \(A\) be a central simple
     \(F\)-algebra.
     We can construct a commutative \(F\)-algebra~\(K\) in the triangular form \(F[x_1,\dots,x_\ell]=F[X_1,\dots,X_\ell]\slash(P_1(X_1),P_2(X_1,X_2),\dots)\), where each \(P_i((x_j)_{j<i},X_i)\) is monic of degree~\hbox{\(>1\)} in \(X_i\) in the polynomial ring \(F[(x_j)_{j<i}][X_i]\), such that \(A_K\simeq \MM_q(K)\).
     As an important corollary, the \(F\)-dimension of~\(A\), which is equal to the \(K\)-dimension of \(A_K\), is always a square: \([A:F]=[A_K:K]=q^2\).
     The integer~\(q\) is called the 
     \emph{degree}\footnote{Bourbaki's terminology is that \(q\) is the ``reduced degree'' of~\(A\), and that \(q^2\) is the ``degree'' of~\(A\).} of the central simple algebra~\(A\) and is denoted by \(\deg_F(A)=\deg(A)\).
   \end{thdef}  
 \end{sloppypar}
% %r
% %:     Remark{rem}
% \begin{remark} \label{rem} 
%   the \(F\)-algebra~\(K\) of Theorem \ref{th3}
%    is obtained . 
% \end{remark}
% %----------- fin remark ---------------------------------- 

%\hum{Appeler le théorème précédent «Constructive Wedderburn theorem»? Dans ma compréhension, c'est plus Theorem~\ref{th3} qui est le «Constructive Wedderburn theorem» que Theorem~\ref{th1}. En tous cas il mérite un nom. }
\begin{theorem}  [Skolem-Noether] \label{thSkNo}
Let \(A\) be central simple and \(\sigma\) an \(F\)-automorphism of~\(A\). We can find an invertible element~\(w\) of~\(A\) such that \(\sigma = \Int(w)\).
\end{theorem}  
The proof below is an example of descent
and corresponds to the cohomological argument in \citealt[Theorem 2.7.2]{GS}.
%We use that the previous theorem has a constructive proof and hence makes sense in a dynamical context.
\begin{proof}
  Let us consider the \(F\)-vector space~\(M\) of elements \(a\) in \(A\) such that \(a x = \sigma(x)a\) for all \(x\) in \(A\). We claim that \(M\) is free of rank \(1\).

  We know that we have a triangular extension \(L = F[a_1,\dots,a_n]\) such that \(A_L\) is a matrix algebra.
  By Proposition \ref{prop5anneau}, \(L\otimes_F M\) is free of rank \(1\) over \(L\), being generated by an invertible element of \(A_L\).

  By faithfully flat descent \citep[Theorem~VIII.6.1]{CACM}, \(M\) is locally free of rank \(1\) over \(F\), i.e.\ free of rank \(1\), since \(F\) is a field. It is thus generated by an element
  of \(A\), which is invertible, since this element is invertible in \(A_L\) and \(L\) is a faithfully flat extension of \(F\).
\end{proof}

%\hum{Parler de rang ou de dimension ?}
  
%%   We have dynamically an extension~\(K\) of~\(F\) such that \(A_K\) is a
%%   matrix algebra \(\MM_q(K)\).
%%   We then write \(\sigma(\vep_i) = \sum_k \sigma_{ik} \vep_k\) (with the \(\sigma_{ik}\)'s in~\(F\)) and we 
%%   extend \(\sigma\) into a \(K\)-automorphism \(\sigma_K\) of \(\MM_q(K)\). In fact, \(\sigma_K\) is a surjective \(K\)-endomorphism. By Proposition \ref{prop5}, \(\sigma_K\) is an inner automorphism given by an \(u = \sum_i u_i\vep_i\) with the \(u_i\) in~\(K\).
  
%% \noindent   Now we consider the \(F\)-linear system of equations \(\sigma(\vep_i)v = v\vep_i\) for \(i=1,\dots, m\) in the \(m\) unknowns \(v\in F^m\simeq A\).
%% Seeing this system as a \(K\)-linear system in the \(m\) unknowns \(v\in K^m\simeq A_K\), it has the solution given by  \(u\in K^m\). Another solution~\(v\) satisfies \(u\vep_iu^{-1}v=v\vep_i\), so \(u^{-1}v\) commutes with all elements of~\(A\), i.e.\  \(u^{-1}v\in K\), or \(v\in Ku\). So the solutions of the system over~\(K\) are the elements of the \(K\)-vector space~\(Ku\). But since the linear system is an \(F\)-system, its solutions in \(K^m\) are obtained from its solutions in \(F^m\) by scalar extension.
%% So the solutions of the system over~\(F\) are the elements of an \(F\)-vector space \(Fw\) (\(w\in A\)) such that \(Kw=Ku\). Moreover \(xw=0\) implies \(xu=0\) which implies \(x=0\), so~\(w\) is left-regular, hence invertible.
%% \end{proof}

%% \hum{Thierry va proposer un autre argument pour ce Skolem-Noether.}

Another proof of the Skolem-Noether theorem will be given in Remark~\ref{rem:baer}.

\subsection{Some consequences} 

% \begin{proposition}\label{prop3}
% If \(A\) is central and the canonical map \(A\otimes_F A\op\rightarrow \MM_m(F)\) is not an isomorphism, then we can find a nonzero noninvertible element in~\(A\).
% \end{proposition} 

% \hum{dire qu'on a une extension fidèlement plate \(F\to L\) avec \(A\) déployée: \(A_L\otimes A_L\op\) est isomorphe à \(\End_L(A_L)\) et donc \(A\otimes A\op\) est isomorphe à \(\End_F(A)\)}

% This is a key result, and it appears in one form or another in all treatments that we have seen (\citealp{Albert}
%   calls it a ``fundamental'' lemma and notices that the proof relies on the first basic results of the theory only).

% \begin{proof}
% If this map  is not an isomorphism, this means that the \(m^2\) elements \(\rL_{\vep_i}\rR_{\vep_j}\) are not free in \(\MM_m(F)\).
%   There is then a nontrivial combination \(\sum \rL_{a_j}\rR_{\vep_j} = 0\) with at least one \(a_j\neq 0\). 
%   If all \(a_j \neq 0\) are invertible then we can assume one \(a_{j_0}\) to be \(1_A\). We have \(\sum a_j x \vep_j = 0\) for all~\(x\). If all \(a_j\) were in the center \(F\vep_1\) of  \(\MM_m(F)\), we would get that the \(\vep_j\) are linearly dependent. So there is a \(j_1\neq j_0\) and a \(y\in A\) such that
% \(a_{j_1} y \neq ya_{j_1}\). We then have \(\sum_j a'_jx\vep_j= \sum_j (a_jy - ya_j) x \vep_j = 0\) for all~\(x\). This nontrivial linear combination has a smaller support: it has less nonzero elements since \(a'_{j_0}=0\). 
%   We keep doing this until we find a nonzero noninvertible element.
% \end{proof}  

%: th2
\begin{theorem}\label{th2}
  If \(A\) is central simple then \(A\otimes_F A\op\) is a matrix algebra.
  More precisely,
the canonical map \(A\otimes_F A\op\rightarrow \MM_q(F)\) is an isomorphism.
\end{theorem}  

\begin{proof}
  This is true if \(A\) is a matrix algebra, and by the splitting theorem \ref{th3} it is true for \(A_L\) with \(L\) a faithfully flat
  commutative \(F\)-algebra. As \(L\)~is a faithfully flat extension of~\(F\), the isomorphism between \(A_L\otimes_L A_L\op\) and \(\MM_q(L)\)
  yields an isomorphism between \(A\otimes_F A\op\) and \(\MM_q(F)\).
  % This holds by induction on the dimension of~\(A\) using the previous results.
% More precisely,
% if the canonical map \(A\otimes_F A\op\rightarrow \MM_m(F)\) is not an isomorphism, we can find a nonzero noninvertible element in~\(A\) and we can write \(A = \MM_q(B)\) where \(B\) is central simple and \(q>1\) (by Wedderburn's theorem). By induction we have that \(B\otimes_F B\op\simeq\MM_r(F)\),
% so \(A\otimes_F A\op\simeq\MM_{qr}(F)\).
\end{proof}  
% \hum{Argument du genre «The determinant of the corresponding isomorphism is an element of~\(F\) invertible in~\(K\), so that it is invertible in~\(F\)» nécessaire?}

\begin{remark}\label{rem:baer}
  This yields another proof of the Skolem-Noether theorem \ref{thSkNo} as noted by \citet[page~590]{Baer} \citep[the argument has been taken over by the reprints of][]{Albert}.
We  use that \(A\otimes_F A\op \simeq\MM_q(F)\). One extends \(\sigma\) into an automorphism of \(\MM_q(F)\) which leaves each element
\(\rR_b\) invariant. By  Proposition \ref{prop5}, this automorphism is given by an element~\(u\) of \(\GL_q(F)\). But then \(u\) commutes with each \(\rR_b\), hence  is of the form \(\rL_a\), and \(\sigma\) is the inner automorphism given by~\(a\).
\end{remark}
%

% The next theorem can be seen as detailing Lemma \ref{lemBasic2} when \(A\) is central simple.
 
 %:     Theorem{thBr2}
 \begin{theorem}[{\citealp[Theorem 1.17]{Albert}}] \label{thBr2} 
 Let \(B\simeq \MM_q(F)\) be a matrix subalgebra of an algebra~\(A\). 
 Then \(A = B \otimes_F C\simeq \MM_q(C)\) with~\(C = \rZ_A(B)\).
 %\footnote{Recall that when \(B\) and~\(C\) are commuting subalgebras of a given algebra~\(A\), the notation \(A = B \otimes_F C\) means that the natural morphism \(B \otimes_F C\to A\) is an isomorphism (Lemma~\ref{lemAlg1}).} 
 \end{theorem}
% %----------- fin theorem ----------------------------- 
% %
 \begin{proof}
   This follows from Lemma \ref{lemBasic2}: consider the \(e_{ij}\in B\) given by the isomorphism \(B\simeq \MM_q(F)\); then \(C=\diag{A}\).
 \end{proof}
% %

% \hum{Theorem \ref{th2} ? Non, c'est certainement Theorem \ref{thBr2} ! Donner son nom au Theorem \ref{th3} }
% %:     Theorem{thBr3}
 \begin{theorem} \label{thBr3}%(lemma 0.5)
 Let  \(A\) be an \(F\)-algebra and \(B\) a central simple subalgebra. Then  \(A = B \otimes_F \rZ_A(B) \).
 \end{theorem}
% %----------- fin lemma -----------------------------------

 \begin{proof}
   By the first splitting theorem \ref{th3}, we have a faithfully flat extension \(L\) of \(F\) such that \(B_L = B\otimes_F L\) is
   an \(L\)-matrix algebra. Also \(\rZ_{A_L}(B_L) = (\rZ_A(B))_L\) and, as the canonical map \(B_L\otimes_L (\rZ_A(B))_L\rightarrow A_L\) is an isomorphism by Theorem \ref{thBr2}, so does the canonical map \(B\otimes_F \rZ_A(B)\rightarrow A\).
 \end{proof}

% %\smallskip \noindent \ttt{ Dans la littérature classique un résultat plus général est obtenu sous l'hypothèse plus faible que \(B\) est simple: la footnote \ref{BlaII-3} fait allusion à un fait similaire.}
% %
%%  \begin{proof}
%%  Let us consider the central simple algebra  \(A \otimes_F B\op\) that contains  \(B \otimes_F B\op\simeq \MM_q(F)\) as    subalgebra. 
%%  By Lemma \ref{lemAlg3} (with \(B\op\) instead of~\(B\)),  \(C\otimes_F F\) is the centraliser of \(B \otimes_F B\op\) in~\(A \otimes_F B\op\). 
%%  So \(A \otimes_F B\op=  C\otimes_F B \otimes_F B\op\) by Theorem \ref{thBr2}. By Lemma \ref{lemAlg2}  (with \(D=B\op\)),
%%  we get \(A = C \otimes_F B\). 
%%  The fact that \([A:F]=[B:F][C:F]\) is an immediate consequence of \(A = B \otimes_F C\). The fact that \(C\) is central simple follows from Proposition \ref{prop4}. Finally, since \(B\) and~\(C\) are central, \(B\otimes_F F\) and  \(F\otimes_F C\) centralise each other in \(B \otimes_F C\)  (see Lemma \ref{lemAlg3}).  
%%  \end{proof}
% %

%: prop4
\begin{proposition}\label{prop4}
  If \(A\otimes_F B\) is central simple then both \(A\) and~\(B\) are central simple. In particular, \(A\otimes_F A\op\) is a matrix algebra \ssi \(A\) is central simple.
\end{proposition}  
\begin{proof}
Let \((\vep_1,\dots,\vep_m)\) and \((\vet_1,\dots,\vet_p)\) be \(F\)-bases of~\(A\) and~\(B\), respectively, with \(\vep_1=1_A\) and \(\vet_1=1_B\). The elements \(\vep_i\otimes \vet_j\) form a basis of \(C=A\otimes_F B\). Any element of~\(C\) can be written as \(\sum_ja_j\otimes \vet_j\) with the \(a_j\)'s in~\(A\).
If \(x\neq 0\) in~\(A\), then from \(1_C\in C(x\otimes 1_B)C\) we deduce \(1_{A}\in AxA\): inspect the equality \((\sum_ja_j\otimes \vet_j)\allowbreak(x\otimes 1_B)(\sum_ja'_j\otimes \vet_j)=1_C\). So \(A\) is simple. In a similar way, \(A\) is central% (or use Lemma \ref{lemAlg3+})
. 
\end{proof}
%

%\noindent N.B.: From Theorem \ref{th2} and Proposition \ref{prop4} we get the following  important characterisation: \(A\) is central simple \ssi the canonical map \(A\otimes_F A\op\rightarrow \MM_m(F)\) is an isomorphism.

%\smallskip \hum{Le corollaire suivant est à enlever.}

In the following corollary, the field~\(K\) containing~\(F\) is not supposed to be finitely generated over~\(F\).

\begin{corollary}\label{corth2} Let \(A\) be an \(F\)-algebra, \(K\) a discrete field extension of~\(F\) and \(A_K=K\otimes_F A\).
Then \(A\) is \(F\)-central simple \ssi \(A_K\) is \(K\)-central simple.
\end{corollary}  

\begin{proof}
  We use the previous Proposition. The canonical map \(A\otimes_F A\op\rightarrow \End_F(A)\) is an isomorphism
  iff the corresponding map \(A_K\otimes_K A_K\op\rightarrow \End_K(A_K)\) is, since it can be expressed by the fact that a given
  determinant is invertible.
\end{proof}
%% \emph{1.} Assume that \(A\) is central simple.
%% \smallskip \noindent We are now ready to prove the difficult implication  that \(A\)  central simple implies \(A_K\) central simple. 
%% The \(K\)-algebra \(A_K\otimes_K {A_K}\op\simeq K\otimes_F(A\otimes_F A\op)\) is isomorphic to \(\MM_m(K)\) since \(A\otimes_F A\op\simeq \MM_m(F)\) by  Theorem~\ref{th2}. Hence \(A_K\) is central simple by  Proposition~\ref{prop4}.

%% \smallskip \noindent \emph{2.} Assume that \(A_K\) is central simple. First we check that \(A\) is simple.
%%  If \(B=A_K\) is simple and \(0\neq x\in A\),  we have to see that \(\vep_1\in AxA=\sum_{j,k} F\vep_k x\vep_j\).
%%  This can be seen as a solution of an \(F\)-linear system whose coefficients are given by the coordinates of~\(x\) and the structure constants of~\(A\).  
%%  But \(\vep_1\in BxB=\sum_{j,k} K\vep_k x\vep_j\). Since  this \(F\)-linear system has a solution in~\(K\), it has a solution in~\(F\). \\
%%  Now we check that \(A\) is central. For \(a\in F\), let \(\delta_a\colon A\to A\)
%%  be defined by \(\delta_a=\rL_a-\rR_a\). The center  \(\rZ(A)\) is the intersection of the kernels \(\Ker(\delta_{\vep_i})\). This is a linear algebra computation, which implies that \(\rZ(A_K)\) 
%%  is the extension of \(\rZ_F(A)\) to~\(K\).
%%  is obtained from  \(\rZ(A)\) by scalar extension from~\(F\) to~\(K\).
%%  Since \(\rZ(A_K)=\vep_1K\), this implies that \(\rZ(A)=\vep_1F\).

\begin{corollary}\label{cor2th2} Let \(A\) and~\(B\) be central simple algebras.
 Then  \(A\otimes_F B\) is central simple.
\end{corollary}  
\begin{proof}
We have \((A\otimes_F B)\otimes_F(A\otimes_F B)\op\simeq (A\otimes_F A\op)\otimes_F(B\otimes_F B\op)\simeq \MM_p(F)\otimes_F\MM_q(F)\simeq \MM_{pq}(F)\). We conclude by Proposition \ref{prop4}. 
\end{proof}
%

%r
%:     Remark{remcor2th2}
\begin{corollary}\label{cor3th2} 
  \(A\otimes_F B\) is central simple \ssi \(A\) and~\(B\) are central simple.
  \(A\otimes_F B\) is central  \ssi \(A\) and~\(B\) are central.
\end{corollary}  

\begin{remark} \label{remcor2th2} 
We have also seen that
%If \(A\) is simple and \(B\) central simple, it is possible to show that \(A\otimes_F B\) is simple. \hum{Pourquoi ? }
if \(K\) is a finite field extension of~\(F\), \(K\) is simple,
but for \(x\in K\setminus F\), the element \(x\otimes 1- 1\otimes x\) is nonzero noninvertible in \(K\otimes_F K\), so \(K\otimes_F K\) is not simple. 
  
\end{remark}
%----------- fin remark ---------------------------------- 

\subsection{Reduced characteristic polynomial} 

Here is yet another use of faithfully flat descent.

\begin{propdef}\label{reduced}
  Let \(A\) be a central simple algebra over \(F\) of dimension \(n^2\). Let \(u\) be an element of \(A\)
and consider the linear map of multiplication by \(u\) and its characteristic polynomial \(C_u\) in \(F[X]\)
of degree \(n^2\). There exists a polynomial \(R_u\) in \(F[X]\) of degree \(n\) uniquely determined by the condition \(R_u^n = C_u\) and we have
  \(R_u(u) = 0\).
\begin{itemize}[nosep]
\item This polynomial~\(R_u\) is called the \emph{reduced characteristic polynomial of~\(u\)}, and will be denoted by \(\Cprd_{A/F}(u)\) or by \(\Cprd(u)\) when the context is clear.
\item  Let \(R_u=X^r+\sum_{k=1}^r(-1)^ks_kX^{r-k}\). The coefficient \(s_1\),  denoted by  \(\Trd(u)\),
is called the \emph{reduced trace of~\(u\) (over~\(F\))}.
\item  The coefficient \(s_r\),  denoted by \(\Nrd(u)\),
is called the \emph{reduced norm of~\(u\) (over~\(F\))}.
\end{itemize} 
\end{propdef}

\begin{proof}
  We know that \(A_L\) becomes of the form \(\MM_n(L)\) for some extension \(L = F[a_1,\dots,a_p]\) where we add successively
formal roots of monic nonconstant polynomials.

We can then consider a matrix decomposition \(e_{ij}(\vec{a})\) of \(A_L\) and write \(u = \sum m_{ij}(\vec{a})e_{ij}\).
We have the characteristic polynomial \(R_u(\vec{a})\) of the matrix \(\varphi(u) = m_{ij}(\vec{a})\) in \(\MM_n(L)\) and the problem
reduces to showing that \(R_u(\vec{a})\), which is a priori in \(L[X]\), is actually in \(F[X]\).

It is enough to look at the case \(p = 1\). So we have \(L = F[a]\) and \(P(a) = 0\), \(P\) monic nonconstant.
We have \(u = \sum m_{ij}(a) e_{ij}(a)\).

We consider \(K = L\otimes_F L = F[a,b]\) where we add a second formal root~\(b\) such that \(P(b) = 0\).

We have two matrix decomposition of \(A_K\) given by \(e_{ij}(a)\) and \(e_{ij}(b)\). Also \(K\) is zero dimensional.

By the matrix version of the Skolem-Noether Theorem, we have \(v\) in \(A_K\) invertible such that \(ve_{ij}(a) = e_{ij}(b) v\).
This implies that the matrces \(m_{ij}(a)\) and \(m_{ij}(b)\) are similar via \(v\) (writing \(v\) as an invertible
matrix in \(\MM_n(K)\)) and so have the same characteristic polynomial. This means that \(R_u(a) = R_u(b)\)
and so \(R_u(a)\) is in \(F[X]\). It also satisfies \(R_u^n = C_u\) and is uniquely determined by this
condition.\footnote{Note that we have \(R_u(a)^n = C_u\) but we cannot conclude from this, since \(F[a]\) may have non trivial nilpotent elements.}
%\hum{Que veut dire reduced?}}
\end{proof}

Such descent argument can be traced back to \citet{Chatelet}, who was using Galois descent, which is here replaced by faithfully flat descent.

%l
%:     Lemma{lemNrd}
\begin{lemma} \label{lemNrd}
\(\Nrd(uv)=\Nrd(u)\Nrd(v)\), and \(u\) is invertible iff \(\Nrd(u)\) is nonzero. \\
In this case \(u^{-1}=\widetilde u\ \Nrd(u)^{-1}\), where \(\widetilde u\) is equal to \(Q(u)\) with \((-1)^{r-1}P(X)=XQ(X)-\Nrd(u)\). 
\end{lemma}
%----------- fin lemma ----------------------------------- 
%
\begin{sloppypar}
  \begin{proof} In the previous proof we have \(\varphi(uv)=\varphi(u)\varphi(v)\); taking determinants we get \(\Nrd(uv)=\Nrd(u)\*\Nrd(v)\).
    The last assertion follows from \(\Cprd(u)(u)=0\).
  \end{proof}
\end{sloppypar}
% 

%% \subsubsection*{Reduced trace and reduced norm}

%% In this paragraph we prove the following fact.

%% %l
%% %:     Lemma{lemRedPc}
%% \begin{propdef} \label{lemRedPc}
%% Let \(A\) be a central simple \(F\)-algebra of degree~\(r\). Then the characteristic polynomial~\(Q\) of (left-multiplication by) an element~\(a\) of~\(A\) is a polynomial  (of degree \(n=r^2\)) which is equal to \(P^r(T)\) for a unique monic polynomial \(P\in F[T]\). Moreover \(P(a)=0\).
%% %d
%% %:     Definition{defiNrd}
%% %\begin{definition} \label{defiNrd}~
%\end{propdef}

%\ttt{ des preuves plus directes des deux corollaires sont-elles possibles? }

\subsection{Separably closed fields}

Using the reduced polynomial, we now prove in a similar way
that a (dynamical)  extension~\(L\) that splits~\(A\) can be constructed as a separable \(F\)-algebra.
This means that \(L\) is obtained by successive additions of formal roots of separable polynomials.
%This implies \emph{separability of \(K/F\)} with the meaning that the discriminant of \(K/F\), i.e.\ the Gram determinant of the trace form 
%\((x,y)\mapsto \mathrm{Tr}_{K/F}(xy)\), is an invertible element in~\(F\) \citep[see e.g.][II-5.33 and II-5.36]{CACM}. 

A discrete field~\(K\) is said to be \emph{separably closed} if any separable polynomial is split. 
%\cha{In characteristic~\(p\), the only irreducible polynomials over separably closed fields are of the form \(X^{p^\ell}-d\)%\citep[see e.g.][Theorem VI-6.3]{MRR}  .}
 
%We follow the argument in \citealt{Bla}. 

\begin{lemma}\label{sepclos}
  If \(K\) is separably closed and \(A\) is a central simple algebra over~\(K\) then \(A\) is split.
  \end{lemma}

\begin{proof}
  It is enough to show that we can find an element \(u\) in \(A\setminus K\) such that \(\Cprd(u)\) is separable, when the \(K\)-dimension of
  \(A\) is \(>1\). Let \((\alpha_{i})\) be a basis
  of \(A\) over \(K\) and consider  the polynomial \(\Cprd(\sum X_{i}\alpha_{i})\) in \(K[X_{i}][T]\). Its discriminant is \(\neq 0\), since
  we can find values \(b_{i}\) in some faithfully flat extension \(L\) of \(K\) such that \(\Cprd(\sum b_{i}\alpha_{i})\) is
  separable,\footnote{\(A_L\) is a matrix algebra over \(L\), with a matrix decomposition \((e_{pq})\), and we take
  \(\sum b_{i}\alpha_{i}\) of the form \(\sum_{p} c_p e_{pp}\) with the~\(c_p\) pairwise distinct.} using the first splitting Theorem \ref{th3}.
  Since \(K\) is separably closed, it is infinite, and hence we can find values \(a_{i}\) {\em in} \(K\) such that the discriminant of
  \(\Cprd(u)\) is \(\neq 0\) for \(u = \sum a_{i}\alpha_{i}\), i.e.\ \(\Cprd(u)\) is separable.
\end{proof}

 Note that it is essential for this argument to use the \emph{reduced} polynomial.

As before, reading dynamically this proof \citep[compare][]{CM}, we obtain the following result and a variation.

\begin{theorem}[second splitting theorem]\label{th4}~
%If \(A\) is central simple then we can find dynamically a finite separable extension~\(K\) of~\(F\) such that \(A_K\) is a matrix algebra \(\MM_r(K)\). More precisely, we find a commutative separable \(F\)-algebra~\(K\) such that \(A_K\simeq \MM_r(K)\). 
  If \(A\) is~a central simple algebra over a field~\(F\),
  we can construct a commutative separable \(F\)-algebra~\(K\) which splits~\(F\).
This extension is obtained in the triangular form 
\[F[x_1,\dots,x_k]=F[X_1,\dots,X_k]/(P_1(X_1),P_2(X_1,X_2),\dots)\] with each \(P_i((x_j)_{j<i},X_i)\) monic and separable in \(F[(x_j)_{j<i}][X_i]\) as a polypnomial  in \(X_i\).
\end{theorem}  

A discrete field~\(K\) is said to be \emph{separably factorial} when any separable polynomial in \(F[X]\) can be factorised into irreducible factors.
This property is inherited by strictly finite separable extensions of~\(K\).

%:     Theorem{th40}
\begin{theorem} \label{th40}
\cha{Let \(F\) be a separably factorial field and \(A\)  a central simple
  \(F\)-algebra.
We can construct a separable field extension~\(K\) of~\(F\) that is strictly finite and splits~\(A\). This extension is obtained in a triangular form.}
 \end{theorem}
%----------- fin theorem ----------------------------- 

%
%

%If \(K\) is a subfield of a separably closed field~\(L\), with \(K\) detachable in~\(L\), then \(K\) is separably factorial (see \cite[Chapter VII]{MRR}).
%If \(F\) is separably factorial, we can refine Theorem \ref{th4} and obtain~\(K\) as
%a field extension of~\(F\).

%%%%%%%%%%%%%%%%%%%%%%%%%%%%%%%%%%%%%%%%%%%%%%%%%%%%%%%%%%%%%%%%%%%%
\section{The Brauer group of a discrete field~\(F\)}

%d
%:     Definition{defiBreq}
\begin{definition} \label{defiBreq}
Two central simple \(F\)-algebras \(A\) and~\(B\) are said to be \emph{Brauer equivalent} if there are~\(m\) and~\(n\) such that \(\MM_m(A)\simeq \MM_n(B)\).
We denote this by \(A\sim_F B\) or \([A]_F^\Br=[B]_F^\Br\) or \([A]=[B]\).  
\end{definition}
%----------- fin definition -------------------------------- 

\noindent N.B.: It is easy to see that \(\bullet\sim_F\bullet\) is an equivalence relation by using \(\MM_p(\MM_q(A))\simeq \MM_{pq}(A)\). 
%l
%:     Lemma{lemBr1}
\begin{lemma} \label{lemBr1}
If \(A\sim_F A'\) and \(B\sim_F B'\) then \(A\otimes_FA'\sim_F B\otimes_FB' \) 
\end{lemma}
%----------- fin lemma -----------------------------------
%
\begin{proof} We have \(\MM_m(A)\simeq \MM_n(A')\) and \(\MM_p(B)\simeq \MM_q(B')\). Since \(\MM_m(A)\simeq \MM_m(F)\otimes_F A\) we get 
\[\MM_{mp}(A\otimes_FB)\simeq \MM_m(A)\otimes_F\MM_{p}(B)\simeq \MM_n(A')\otimes_F\MM_{q}(B')\simeq \MM_{nq}(A'\otimes_FB').\qedhere\]
\end{proof}

As a corollary, the equivalence classes of central simple algebras form a group with neutral element \(1=[F]\), the inverse of \([A]\) being \([A\op]\) by Theorem \ref{th2}. It is called the \emph{Brauer group of~\(F\)} and is denoted by~\(\Br(F)\).

Note the following simplification rule:  if \(A\), \(B\),~\(C\) are central simple algebras such that \(A\otimes_FC\sim_F  B\otimes_FC\), 
this can be written \([A]_F[C]_F = [B]_F[C]_F\), so \([A]_F=[B]_F\), i.e.\ \(A\sim_F B\).

\medskip

%% The next theorem is a generalisation of Proposition \ref{prop5}: when we take \(A=B=F\) and two isomorphisms \(\MM_m(A)\to C\) and  \(\MM_m(B)\to C\), we get Proposition \ref{prop5} since the proof ends at the first step (\(F\) is a division algebra over itself!). It should be interesting to compare the concrete computations involved in the two proofs. \hum{Remarque toujours pertinente ?}% Note that \(\sigma(y_0y^T)z\neq 0\) in the proof of Proposition \ref{prop5} is similar to \(d_{pq}\neq 0\) in the proof of Theorem \ref{thBr1}.

\begin{theorem} \label{thBr10} Let \(A\) be a division algebra and~\(B\) a central simple algebra.
If \( \MM_m(A)\simeq C\simeq \MM_n(B)\), then \(n\)~divides~\(m\) and \(B\simeq \MM_{m/n}(A)\). 
%
% \\ More precisely, consider the matrix decompositions \((e_{ij})_{1\leq i,j\leq mm'}\)  and \((f_{ij})_{1\leq i,j\leq nn'}\) in~\(C\) corresponding respectively to  isomorphisms \(C\simeq \MM_{mm'}(A')\) and \(C\simeq \MM_{nn'}(B')\). Then there is an invertible \(g\in C\) such that  \(f_{ij}= ge_{ij}g^{-1}\) for all \(i,j\). In particular, the
% automorphism \(x\mapsto gxg^{-1}\) of~\(C\) sends \(e_1Ce_1\), which is isomorphic to \(A'\),  to \(f_{1}Cf_{1}\), which is isomorphic to \(B'\).
%

\end{theorem}

\begin{proof}
The \(F\)-algebra~\(C\) is isomorphic to \(\End_A(A^m)\) and to \(\End_B(B^n)\). Assume in a first step that \(A\)~is a division algebra. Then \(A^m\)~is a right \(A\)-vector space. Consider the matrix decomposition \((f_{ij})_{1\leq i,j\leq n}\) for~\(C\) corresponding to the isomorphism \(C\simeq\MM_n(B)\): it produces a decomposition of~\(A^m\) into the direct sum of~\(n\) pairwise isomorphic right \(A\)-vector spaces. Therefore \(n\)~divides~\(m\) and \(B\) is isomorphic to~\(\MM_{m/n}(A)\).
\end{proof}
%:     theorem{thBr1}
\begin{corollary} \label{thBr1} Let \(A\) and~\(B\) be central simple algebras.
\begin{enumerate}[nosep]
\item If \( \MM_m(A)\simeq C\simeq \MM_n(B)\) then there exist \(m'\), \(A'\), \(n'\), \(B'\) such that \(A\simeq \MM_{m'}(A')\), \(B\simeq \MM_{n'}(B')\), \(mm'=nn'\) and \(A'\simeq B'\). 
%
% \\ More precisely, consider the matrix decompositions \((e_{ij})_{1\leq i,j\leq mm'}\)  and \((f_{ij})_{1\leq i,j\leq nn'}\) in~\(C\) corresponding respectively to  isomorphisms \(C\simeq \MM_{mm'}(A')\) and \(C\simeq \MM_{nn'}(B')\). Then there is an invertible \(g\in C\) such that  \(f_{ij}= ge_{ij}g^{-1}\) for all \(i,j\). In particular, the
% automorphism \(x\mapsto gxg^{-1}\) of~\(C\) sends \(e_1Ce_1\), which is isomorphic to \(A'\),  to \(f_{1}Cf_{1}\), which is isomorphic to \(B'\).
%
\item In particular, 
\begin{enumerate}[nosep]
\item \([A]_F=1\) \ssi \(A\) is split.
\item \cha{If \(A\sim_F B\)  and \([B:F]=[A:F]\) then \(A\simeq B\)}.
\item Assume that \(A=\MM_r(D)\) where \(D\) is a division algebra. Then \(B\sim_F A\) \ssi there exists an~\(s\) such that \(B\simeq \MM_s(D)\), and \(B\simeq A\) \ssi \(B\sim_F A\) and \(r=s\).
\end{enumerate}

\end{enumerate}
\end{corollary}
%----------- fin theorem ----------------------------------- 
% Note that when \(A\) and~\(B\) are division algebras, the theorem implies that
% \(m'=n'=1\), \(m=n\) and \(A\simeq B\). This is the classical form of Theorem~\ref{thBr1}
% in the classical litterature. We will follow the proof  given in this case in \citealt[Theorem 4.1]{Albert}.
%
\begin{proof}
% See Lemma \ref{lemBasic2}.
  \emph{1}. Let us consider the proof of Theorem~\ref{thBr10}. We have used the hypothesis that \(A\)~is a division algebra in the representation of~\(A^m\) as the direct sum of vector spaces all isomorphic to~\(B\). If this representation fails, it is because we find a nonzero nonregular element in~\(A\) and we know by Theorem~\ref{th1} that \(A\) is isomorphic to some~\(\MM_{m'}(A')\) with \(m'>1\) and \([A':F]<[A:F]\). So we can replace in our hypothesis \(A\) by~\(A'\) and \(m\) by \(mm'\). This kind of problem can only happen a finite number of times.

 \medskip\noindent \emph{2}. Left to the reader.
\end{proof}
%
%l
%:     Lemma{lemBr2}
\begin{lemma} \label{lemBr2}\label{corthBr1}
Let \(A\) and~\(B\) be central simple \(F\)-algebras, and \(K\)  an extension of~\(F\). If \(A\sim_FB\), then \(A_K\sim_KB_K\). In particular, if \(A\sim_FB\), then \(K\) splits~\(A\) \ssi \(K\) splits~\(B\).
\end{lemma}
%----------- fin lemma ----------------------------------- 
%
\begin{proof}
If \(\MM_m(A)\simeq \MM_n(B)\) as \(F\)-algebras, then 
\[\MM_m(A_K)
\simeq K\otimes_F\MM_m(A)
\simeq K\otimes_F\MM_n(B)
\simeq \MM_n(B_K)
\] 
as \(K\)-algebras.
\end{proof}

This gives a natural  group morphism \(\Br(F)\to\Br(K)\).

\begin{lemma} \label{lemBr4}  %(lemma 0.2)\\
Let  \(A\) be a central simple \(F\)-algebra and \(L\) a field extension that is strictly finite over~\(F\). If~\(L\) splits~\(A\), then we can construct an algebra \(B\sim_FA\) such that \(L\) is (isomorphic to) a maximal 
subfield of~\(B\). Moreover \([B:F]=[L:F]^2\).
\end{lemma}
%----------- fin lemma ----------------------------------- 
%
\begin{proof}
  Suppose that \(L\) is a field extension of~\(F\) of degree~\(n\) splitting~\(A\): we have a matrix decomposition 
  \((e_{ij})_{1\leq i,j\leq m}\) of \(A_L\). Then \(V = e_{11}A_L = Le_{11}+\dots+Le_{1m}\) is an \(F\)-vector space of dimension~\(mn\).

  Let us suppose that \(A\)~is a division algebra: then \(V\)~may also be viewed
  as a right \(A\)-vector space; let \(r\)~be its dimension% and \(v_1,\dots,v_r\) an \(A\)-base
  . Then \(\dim_FV =r\dim_FA=rm^2\), so that \(rm^2=mn\) and \(rm = n\).% and \(r\mathrel|n\)

  The field~\(L\) may be viewed as a subalgebra of~\(B=\End_A(V)\), and \(B\) has dimension \(r^2m^2 = n^2\) over \(F\).

  In the general case, we can follow this argument dynamically: if we encounter an element in \(A\) that is noninvertible
  when trying to find an \(A\)-basis of \(V\), we can apply Theorem \ref{th1} and replace \(A\) by a simpler central simple
  \(F\)-algebra.
  %Therefore \(B\)~cannot contain elements of degree~\(>n\).
\end{proof}

\subsection{Involutions of central simple algebras}

Reference: \citealt[Chapter 8]{scharlau}.

We assume \(\Char (K) \neq 2\).

An \emph{involution} of a \(K\)-algebra \(A\) is a  map \(J:A\to A\) satisfying \(J(a+b)=J(a)+J(b)\), \(J(ab)=J(b)J(a)\) and \(J^2=\Id_A\).
We use the notation \(J(a)=a^J\).   

We let \(A_J^+=A^+:=\sotq{x\in A}{x^J=x}\) and \(A_J^-=A^-:=\sotq{x\in A}{x^J+x=0}\).

We have \(A^+=\sotq{\frac 1 2(x+x^J)}{x\in A}\) and \(A^-=\sotq{\frac 1 2(x-x^J)}{x\in A}\). So \(A=A^+\oplus A^-\).

\smallskip If \(A\) is central, then \(J(K)=K\) and \(F=K \cap A^+\) is a subfield with \(F=K\) or \([K:F]=2\).  If \(F=K\) we say that the involution is  \emph{of the first kind}; if \([K:F]=2\) we say that  the involution is  \emph{of the second kind} or \emph{hermitian}. In each case, the product of two involutions of the same kind is a \(K\)-automorphism. If \(A\) is central simple, this product is an inner automorphism \(\Int(c)\).

\smallskip We describe now the usual constructive classification of involutions of the first kind of central simple algebras. We begin with the case of~\(\MM_n(K)\).

%p
%:     lemma{leùInv1stkind}
\begin{lemma} \label{lemInv1stkind}
Let \(A=\MM_n(K)\) and \(J\) an involution of the first kind (i.e.\ a \(K\)-linear involution). Recall that \(A=A_J^+\oplus A_J^-\). 
%The transposition \(t:A\to A\) is an involution. 
There exists \(b\in A^\times \) such that \(J=\Int(b)\circ t\) and \(b^t=\pm b\). There are two cases. 
\begin{enumerate}[nosep] 
\item 
In the first case, the matrix~\(b\) is symmetric and  \([A_J^+:K]=n(n+1)/2\) (\(J\) is said to be \emph{orthogonal}). 
\item 
In the second case, the matrix~\(b\) is skew-symmetric, \(n\) is even and  \([A_J^+:K]=n(n-1)/2\)  (\(J\) is said to be \emph{symplectic}). Moreover, %\(A_J^+=\sotq{a\in A}{ab=-(ab)^t}\) 
 the characteristic polynomial \(\Cp(a)(T)\) of any matrix \(a\in A_J^+\) is equal to the square of the \emph{Pfaffian characteristic polynomial} \(\Cpf(a)(T)%=\Cpf_a(T)
\) of~\(a\): it is the unique monic polynomial in \(K[T]\) such that \({\Cpf(a)}^2=\Cp(a)\). 
\end{enumerate}
\end{lemma}
%----------- fin lemma ----------------------------- 

%
\begin{proof}
When \(A=\MM_n(K)\), the transposition \(t:A\to A\) is an involution and \(J\circ t=(t\circ J)^{-1}\) is an automorphism of~\(A\), equal to \(\Int (b)\) for some \(b\in A^\times\); in other words \(J(a^t)=b a b^{-1}\), or \(J(a)=b\ a^t \ b^{-1}\).
Writing \(J^2=\Id_A\) we find 
\[
x=(x^J)^J=b (x^J)^t b^{-1}=b (b x^t b^{-1})^t b^{-1}   
= v^{-1} x v \;\hbox{ with }v=b^tb^{-1}.
\] 
Since \(\Int(v)=\Id_{A}\) we have \(b^tb^{-1}=\lambda \in K\), i.e.\ \(b^t=\lambda  b\), which gives \(b=\lambda b^t=\lambda^2b\). \\
So \(\lambda =\pm1\) and \(J=\Int(b)\circ t\) with \(b^t=\pm b\). Note that \(b\) is the matrix of some nondegenerate bilinear form on~\(K^n\), 
\((x,y)\mapsto x^t b y\),\footnote{Here \(x\) and~\(y\) are viewed as column vectors.} which is symmetric or skew-symmetric. \\
If~\(b\) is symmetric (\(J\) is orthogonal),  we get \(A_J^+=\sotq{a\in A}{(ab)^t=ab}\), so \([A_J^+:K]=n(n+1)/2\).
\\
If \(b\) is skew-symmetric (\(J\) is symplectic),  \(n\) has to be even. We get \(A_J^+=\sotq{a\in A}{(ab)^t=-ab}\), so \([A_J^+:K]=n(n-1)/2\). 
\citet[Proposition (2.9)]{BOI} provide the rest of the lemma.
%
%Moreover \citet[Proposition I-2.9]{BOI}  show that the characteristic polynomial \(\Cp(a)(T)\) of any matrix \(a\in A_J^+\) is equal to the square of the so called \emph{Pfaffian characteristic polynomial} \(\Cpf(a)(T)\) of~\(a\). \(\Cpf(a)\) is the unique monic polynomial such that \({\Cpf(a)}^2=\Cp(a)\).
%In the first case \(A_J^+\) is the subvector space of matrices becoming symmetric after the change of basis associated to~\(b\), so . In the second case,  \(n\) is even,  \(A_J^+\) is the subvector space of matrices becoming skew-symmetric after the change of basis associated to~\(b\), so  moreover \(A_J^+=\sotq{a\in A}{ba=ab}\)
\end{proof} 

In the two cases, for a given~\(J\), the matrix~\(b\) and the nondegenerate (symmetric or skew-symmetric) bilinear form defined by~\(b\) are well-defined up to a scalar multiplicative factor.

Now we state the general theorem.

%p
%:     Proposition{propInv1stkind}
\begin{proposition} \label{propInv1stkind}
Let \(A\) be a central simple \(K\)-algebra of degree \(n>1\) and~\(J\) an involution of the first kind (i.e.\ a \(K\)-linear involution). Recall that \(A=A_J^+\oplus A_J^-\).
There are two cases. 
\begin{enumerate}[nosep] 
\item In the first case, \([A_J^+:K]=n(n+1)/2\) (\(J\) is said to be orthogonal). 

\item In the second case, \(n\) is even and  \([A_J^+:K]=n(n-1)/2\)  (\(J\) is said to be symplectic). Moreover the reduced characteristic polynomial \(\Cprd(a)(T)\) of any  \(a\in A_J^+\) is equal to the square of the \emph{Pfaffian characteristic polynomial} \(\Cpf(a)(T)%=\Cpf_a(T)
\) of~\(a\); \(\Cpf(a)(T)\) is the unique monic polynomial in \(K[T]\) such that \({\Cpf(a)}^2=\Cprd(a)\).  
%
%\item  
\end{enumerate}
\end{proposition}
%----------- fin proposition ----------------------------- 
%
\begin{proof}
  By faithfully flat descent, as in Proposition \ref{reduced}.
\end{proof}

%% In the general case the central simple \(K\)-algebra~\(A\) is split by a commutative \(K\)-algebra~\(L\). 

%% Assume first that \(L\) is a field. Then the proposition is satisfied for \(A_L\) and the extension \(J_L\) of~\(J\) to~\(A_L\). These assertions transfer clearly to~\(J\). 

%% In the general case, we cannot force constructively~\(L\) to be a field, but dynamically we obtain the result as if \(L\) was a field: when an obstruction in the proof appears with a nonzero noninvertible element \(z\in L\), we replace~\(L\) by \(L/\!\gen{z}\). 

%

%%%%%%%%%%%%%%%%%%%%%%%%%%%%%%%%%%%%%%%%%%%%%%%%%%%%%%%%%%%%%%%%%%%%
\section{Constructive rereading of a proof by Becher} \label{secBecher}

%\hum{rajouter un commentaire sur le besoin de considérer des polynômes non séparables.}
 
\smallskip
 We denote by \(\Br_2(F)\) the subgroup of \(\Br(F)\) of elements \([A]\) of order \(2\) (i.e.\ such that
 \([A]^2=1\)). 
 
\smallskip   In this final section, \(F\) is a discrete field of characteristic \(\neq 2\). We shall prove two constructive versions of the following result in \citealt[Theorem]{Becher}.

%t
%:     Theorem{thBecher}
\begin{theorem} \label{thBecher}
Let \(F\) be a discrete field of characteristic \(\neq 2\). Every element of \(\Br_2(F)\) is split by a finite extension of~\(F\) obtained by a tower of quadratic extensions. 
\end{theorem}
%----------- fin theorem ----------------------------- 

\subsection{Some results about quaternion algebras}

When \(F\) has characteristic \(\neq 2\), a \emph{quaternion algebra} is defined as the 4-dimensional \(F\)-algebra \(\rh_F(a,b)=\rh(a,b)\) with basis \((1,\alpha,\beta,\gamma)\), the multiplication table being determined by
\[
\alpha^2=a,\,\beta^2=b,\, \gamma=\alpha\beta=-\beta\alpha\text,
\]
where \(a,b\in F^\times \). This is a central simple algebra. 

We have \(\alpha\gamma=-\gamma\alpha\), \(\beta\gamma=-\gamma\beta\) and \(\gamma^2=-ab\).
It is generated as an \(F\)-algebra by \(\alpha\) and \(\beta\).  
If \(q=x+y\alpha+z\beta+w\gamma\in\rh(a,b)\) we let 
\[\conj{q}=x-y\alpha-z\beta-w\gamma\hbox{  and  }\rN(q)=q\conj{q}=x^2-ay^2-bz^2+abw^2\in F.
\] 
The map \(q\mapsto \conj{q}\) is an involution of \(\rh(a,b)\) and \(\rN(q_1q_2)=\rN(q_1)\rN(q_2)\). 
Moreover \(q\) is invertible \ssi \(\rN(q)\neq 0\).  A \emph{pure quaternion} is a quaternion such that \(x=0\). It is characterised by \(q\notin F\) whereas \(q^2\in F\) (or \(q=0\)), or also by
\(\conj{q}=-q\), or also by \(\rN( q)=-q^2\).

The following theorem recalls basic results which are given with a constructive proof in most textbooks \citep[e.g.][Chapter 1]{GS}.

%%%%%%%%%%%%%%%%%%%%%%%%%%%%%%%%%%%%%%%%%%%%%%%%%%%%%%%%%%%%%%%%%%%%
%:     Theorem{thquaternions1}
\begin{theorem}[basics about quaternion algebras] \label{thquaternions1}~
\begin{enumerate}[nosep]
\item \(\rh(a,b)\simeq \rh(u^2a,v^2b)\) \((u,v\in F^\times)\).
\item \(\rh(a,b)\simeq \rh(b,a)\simeq  \rh(a,-ab)\simeq  \rh(b,-ab)\).
\item \( \rh(1,b)\) is split (isomorphic to \(\MM_2(F)\)), so \( \rh(u^2,b)\) is split. 
\item \(\rh(a,1-a)\simeq \rh(1,1)\) if \(a\neq 0,1\).
\item \(\rh(a,b) \otimes_F  \rh(a',b)\simeq \rh(aa',b) \otimes_F \MM_2(F)\). 
\item \(\rh(a,b) \otimes_F  \rh(a,b')\simeq \rh(a,bb') \otimes_F \MM_2(F)\). 
\item The following are equivalent.
\begin{itemize}[nosep]
\item \(\rh(a,b)\) is split.
\item  \(\rN\colon \rh(a,b)\to F\) has a nontrivial zero.
\item  the conic \(ax^2+by^2-z^2=0\) has an \(F\)-rational point (in \(\PP^2(F)\)).
\end{itemize}

\item The following are equivalent.
\begin{itemize}[nosep]
\item \(\rh(a,b)\) is a division algebra.
\item  \(\rN\colon \rh(a,b)\to F\) has no nontrivial zero.
\item  the conic \(ax^2+by^2-z^2=0\) has no \(F\)-rational point (in \(\PP^2(F)\)).
\item  \(\rh(a,b)\) is not split.
\end{itemize}
\item \label{thquaternions1-9}If \(a\) is not a square in~\(F\), the following are equivalent.
\begin{itemize}[nosep]
\item  \(\rh(a,b)\) is split.
\item \(\exists c,d \in F,\, a=c^2-d^2b\).
\end{itemize}
\end{enumerate} 
\end{theorem}
%----------- fin theorem ----------------------------- 

%l
%:     Lemma{lemquaternion1}
\begin{lemma} \label{lemquaternion1} %(lemma 0.3).
A central \(F\)-algebra~\(A\)  of dimension~\(4\) is a quaternion algebra (\/\(\Char(F)\neq 2\)). 
\end{lemma}
%----------- fin lemma ----------------------------------- 
%
\begin{proof}
Let \(z\in A\setminus F\); \(F[z]\) is a commutative algebra of \(F\)-dimension \(d=2\) since \(d\) divides~\(4\) and is neither~\(1\) nor~\(4\). Since \(\Char(F)\neq 2\) we have an~\(x\) such that \(F[z]=F[x]\) and \(x^2=a\in F\). The linear map \(\rho_x\colon v\mapsto x^{-1}vx\) is an automorphism of order \(\leq 2\). But \(\rho_x\neq \Id_A\) since \(xv\neq vx\) for some~\(v\). So there is an eigenvector~\(y\) such that \(\rho_x(y)=-y\) and we get an \(F\)-basis \((1,x,y,xy)\). Then \(y^2x=y(-xy)=-(yx)y=xy^2\), so \(y^2=b\) is in the center~\(F\). We conclude that \(A\simeq\rh(a,b)\). 
\end{proof}
%

%In the following, except for the second constructive version of Becher's theorem, we assume that \(F\) is a fully factorial field, or even, since this has the same consequences for algorithms, that \(F\) is contained in an algebraically closed discrete field~\(L\) such that \(F\) is detachable in~\(L\). 

\subsection{First step}

This step corresponds to \citealt[Lemma 1]{Becher} (Lemma \ref{lem1} below), and the proof will follow its main idea.

%c
%:     context{cont1step}
\begin{context} \label{cont1step}
In this first step we consider a quadratic field extension~\(K\) of~\(F\). Since \(\Char(F)\neq 2\)
we can write
\(K=F[\delta]\) with \(\delta^2=g\in F\) and \(\delta\notin F\). %See \citealt[Section 8.9]{scharlau}. 
\end{context}
%--------- fin corollary ------------------------------- 
 
We have the \(F\)-automorphism \(\sigma\colon x+y\delta\mapsto x-y\delta\) of~\(K\) (\(x,y\in F\)). We use the notation \(\bar w=\sigma(w)\).

We consider an arbitrary quaternion \(K\)-algebra \(A = \rh_K(a+b\delta,c+d\delta)\) with \(a,b,c,d\in F\), \((a,b)\neq (0,0)\), and \((c,d)\neq (0,0) \).  We have \(a^2\neq b^2g\) and \(c^2\neq d^2g\)
(since \(\delta\notin F\)). We consider the \emph{dual} quaternion \(K\)-algebra \(\bar{A} = \rh_K(a-b\delta,c-d\delta)\). The generators of~\(A\) and \(\bar A\) satisfy the following equalities.
\[ 
\begin{array}{ccccc} 
u^2 = a + b\delta & v^2 = c + d\delta & uv = -vu.     \\[.3em] 
\bar{u}^2 = a - b\delta & \bar{v}^2 = c - d\delta & \bar{u}\bar{v} = -\bar{v}\bar{u}.   
 \end{array}
\]

 The \(K\)-semilinear map \(G: w\mapsto \bar w,\,A\to \bar A\) defined on the \(K\)-basis \((1_A,u,v,uv)\) by 
\[1_A\mapsto 1_{\bar A},\,  u \mapsto \bar u,\,  v\mapsto \bar v\hbox{ and }uv\mapsto \bar u \bar v\] 
 is an \(F\)-isomorphism from~\(A\) to \(\bar A\), i.e.\ when we identify~\(K\) with its image in~\(A\), it is an isomorphism of \(F\)-algebras and \(\ov{zw}=\bar{z}\bar w\) if \(z\in K\)
 and \(w\in A\). 

Let us consider the central simple \(K\)-algebra \(B:=A\otimes_K \bar{A}\).
The automorphism \(I_B\) of the \(F\)-vector space~\(B\) defined
by \(e\otimes f\mapsto G^{-1} (f)\otimes G(e)\) is a ring automorphism of order 2
which satisfies \(I_B(z)=\bar{z}\) if \(z\in K\).

Since \(I_B^2=\Id_B\), we have \(B=\Ker(I_B-\Id_B)\oplus \Ker(I_B+\Id_B)\).
Moreover, since \(I_B(\delta)=-\delta\), the \(K\)-linear map \(w\mapsto \delta w\) permutes the two summands.
So the invariant elements
form a sub-\(F\)-vector space of dimension \(16\): \(T=\Ker(I_B-\Id_B)\) is a sub-\(F\)-algebra such that \(B=T\oplus \delta T=K\otimes_F T\). By Corollary~\ref{corth2},  \(T\) is a central simple algebra of degree~\(4\) over~\(F\); it is called the \emph{corestriction of~\(A\) over~\(F\)}.

\begin{lemma}\label{lem11}
  \(T\) contains a quaternion subalgebra, and hence is the tensor product of
  two quaternion \(F\)-algebras.
\end{lemma}

\begin{proof}
We take \(x = v\bar{v}\) and \(y = (u+\bar{u})(cb-ad + d u\bar{u} + b v \bar{v})\)
and we verify directly that \(x^2\) and~\(y^2\) belong to~\(F\) and \(xy=-yx\).
The consequence is obtained by applying  Theorem \ref{thBr3} and Lemma~\ref{lemquaternion1}.
\end{proof}

%%%%%%%%%%%%%%%%%%%%%%%%%%%%%%%%%%%%%%%%%%%%%%%%%%%%%%%%%%%%%%%%%%%%
\begin{lemma}\label{lem12}
  If %\(F\) satisfies Property \(\Adeux\) and 
  \(A\otimes_K\bar{A}\) is split, then \(A\) is the scalar extension of a quaternion
  \(F\)-algebra. % and  is therefore split.
\end{lemma}

%\hum{Je ne suis pas convaincu de l'usage de « extension » ici. « \(A\) is obtained by scalar extension of a quaternion
%  \(F\)-algebra » ? Et préciser dans le §~1 que \(A_K\) est une « scalar extension » par une algèbre commutative \(K\) ?}

%\hum{l'hypothèse \(A\otimes_K\bar{A}\) déployée est-elle nécessaire? Kader Bingol évoque ceci: \(K\) extension de degré~\(2\) de \(F\); \(A\) centrale simple sur K donne par corestriction \(\operatorname{cor}_{K/F}A=T^{K/F}\) et en retour on a \(T_K\simeq A\otimes_K\bar{A}\) et \(\bar A=\operatorname{cor}_{K/F}A\). Si on a \(\Adeux\), alors \(\Gamma=H_1\otimes_F H_2\) et \(T_K=A\otimes_K\bar{A}\) est déployée.}

%\hum{Dans la preuve ci-dessous on renvoie au résultat suivant, «Theorem \ref{thBr2}»: si \(B\simeq \MM_q(F)\) est une sous-algèbre déployée d'une algèbre centrale simple~\(A\), alors \(A = B \otimes_F C\simeq \MM_q(C)\) avec \(C=\rZ_A(B)\). Il y a certainement un argument direct.}
\begin{proof}
  Since \(A\otimes_K \bar{A}\) is split, it is isomorphic to \(A\otimes_K A\op\)  (Theorem \ref{th2})
  and we get a \(K\)-algebra isomorphism  \(J\colon A\rightarrow \bar{A}\op\).
  Composing~\(J\) with the \(F\)-algebra isomorphism \(\bar{A}\rightarrow A\) we get
  an \(F\)-linear isomorphism \(H:A\rightarrow A\op\) satisfying \(H(z) = \bar{z}\) if \(z\in K\) and \(H(xy) = H(y)H(x)\) for all \(x,y\in A\).
   Consider   the canonical \(K\)-linear involution \(I\colon A\rightarrow A\) defined by
  \(I(u) = -u,~I(v) = -v,~I(uv) = -uv=I(v)I(u)\).
  It is clear that \(C := \Ker(H-I)\) is a  sub-\(F\)-algebra of~\(A\) and \(\delta C\subseteq \Ker(H+I)\).\\
  Assume \(HI=IH\); then \(C\oplus \delta C =  A\)  since in this case \(0=(H+I)(H-I)\) and hence \(A=\Ker(H-I)\oplus\Ker(H+I)\).
  We thus get that \(A=K\otimes_F C\) is an extension of \(C\) which is a quaternion \(F\)-algebra by Lemma~\ref{lemquaternion1}.
\\
It remains to prove that \(HI=IH\). For~\(x\) in \(A\setminus K\), the element \(I(x)\) is the unique element in~\(A\)
such that both \(x+I(x)\) and \(xI(x) = I(x)x\) are in~\(K\).\footnote{The commutative \(K\)-algebra \(K_1=K[x]\) is a quadratic field extension of~\(K\) and the conjugate of~\(x\) for the corresponding \(K\)-automorphism of \(K_1\) is the unique  element~\(y\) of \(K_1\) satisfying \(x+y\in K\) and \(xy=yx\in K\). In fact, \(K=\rZ_A(K)\), so~\(y\) is also the unique element of~\(A\) satisfying these properties.
Finally, a direct computation shows that \(x+I(x)\in K\) and \(xI(x)=I(x)x\in K\).}
Then \(H(x+I(x)) = H(x) + H(I(x))\in K\) and \(H(x) H(I(x))= H(I(x)x) = H(xI(x)) = H(I(x))H(x)\in K\), and
so \(H(I(x)) = I(H(x))\). For~\(x\) in~\(K\) we have \(H(I(x)) = \sigma(x) = I(H(x))\).
\end{proof}

%\hum{Attention, ici il est fait appel au double centralisateur qu'on a enlevé de l'article (anciennement Theorem~2.12).}
%d
%:     Definition{defiA2}
\begin{definition}[\citealt{Lam}, Chapter XI, Section 4] \label{defiA2} We say that a field~\(F\) satisfies
 \emph{Property \(\Adeux\)} if \(F\) splits every quaternion \(F\)-algebra. 
\end{definition}
%----------- fin definition -------------------------------- 

%\hum{citer Lam précisément.}

%%%%%%%%%%%%%%%%%%%%%%%%%%%%%%%%%%%%%%%%%%%%%%%%%%%%%%%%%%%%%%%%%%%%
%: lem1
\begin{lemma}\label{lem1} 
If  \(F\) satisfies Property \(\Adeux\) then so does~\(K\).
\end{lemma}

Note also that this can be interpreted as follows: if any conic over~\(F\) has a rational point, then the same holds for~\(K\).

\begin{proof}
We consider a quaternion algebra~\(A\) over~\(K\).
Using Lemma \ref{lem12} it is enough to show that~{\(B=A\otimes_K \bar{A}\)} is split. Lemma \ref{lem11} says that \(T\) is the tensor product of two quaternion \(F\)-algebras, so \(T\) is split,
and the result follows from the fact that \(B = K\otimes_F T\).
\end{proof}

 \subsection{Second step}

 This corresponds to \citealt[Proposition]{Becher}, and the proof can be taken as it stands.

%: prop
 \begin{lemma}\label{lem2}
   Let \(K\) be a discrete field 
   extension of~\(F\) with a primitive element~\(t\)  of degree \(N>2\). \\
   Assume that any polynomial
   of degree \(<N\) has a root in~\(F\); then \(K\) satisfies Property \(\Adeux\).
 \end{lemma}

 \begin{proof}
Note that the hypothesis implies in particular that any element of~\(F\) is a square. 
   We write \(K = F[t]\). Any element of~\(K\) can be written in the form \(f(t)\) with~\(f\)  of degree \(<N\), 
and hence can be written as \(f(t)=c(t-a_1)\cdots (t-a_k)\)
   with \(c\)  a square. Moreover, any quaternion algebra \(\rh_K(t-a,t-b)\) is split since \(b-a\)
   is a square: the conic \((t-a)x^2+(t-b)y^2-z^2\) has a nontrivial zero \((1,\iota,u)\) where \(\iota^2=-1\) and \(u^2=b-a\).
   It follows that in general \(\rh_K(f(t),g(t))\) is split: if \(g(t)=d(t-b_1)\cdots (t-b_\ell)\),  writing
   \((u,v)\) for  \([\rh_K(u,v)]^\Br\) we get\footnote{Note that Theorem \ref{thquaternions1}
   shows that \((uu',v)=(u,v) (u',v)\) and \((u,vv')=(u,v) (u,v')\).} 
   \[(f(t),g(t))=\prod\nolimits_{j=1}^\ell(c,t-b_j) \prod\nolimits_{i=1}^k(t-a_i,d) \prod\nolimits_{i=1}^k\prod\nolimits_{j=1}^\ell(t-a_i,t-b_j) \]
   and all the terms in the products vanish.
 \end{proof}
 
%% Now we give an algorithmic consequence which relates to Lemma \ref{prop} as Corollary \ref{remcor1} does to  Lemma~\ref{lem1}.
 
%% %c
%% %:     Corollary{remprop}
%% \begin{corollary} \label{remprop}
%%    Let \(K\) be a discrete field 
%%    extension of~\(F\) with a primitive element~\(t\) of degree \(N>2\). 
%%    In order to split a given finite number of quaternion algebras by~\(K\),
%%    it suffices to extend~\(F\) by adding a finite number of roots of polynomials of degree \(<N\).
%% \end{corollary}
%--------- fin corollary ------------------------------- 
 
\subsection{Third step}

 This step corresponds to \citealt[Lemma 4]{Becher}, but the proof is different.

 \begin{lemma}\label{lem41}
 Assume that every element of~\(F\) is a square. If \(A\in\Br_2(F)\)  then \(A\) is split or it has a symplectic involution. 
 \end{lemma}
\noindent N.B.: In the last case, \(\deg(A)\) is even. So, if \(\deg(A)\) is odd, \(A\) is split.

 \begin{proof} Let \(N:=\deg(A)\). We use an induction on~\(N\), the case \(N=1\) being trivial. In the case \(N=2\) we have a quaternion algebra (Lemma \ref{lemquaternion1}) and it is split (Theorem \ref{thquaternions1} item \emph{3}). Now we assume \(N>2\). 
The hypothesis \([A]^2=1\) in \(\Br(F)\) means that \(A\otimes_F A\) is split (Theorem \ref{thBr1} item \emph{2a}), hence isomorphic to \(A\otimes_F A\op\) and we have an isomorphism
   \(J\colon A\rightarrow A\op\) \cha{(Theorem \ref{thBr1} item \emph{2b})}.  So \(J^2\) is an automorphism of the \(F\)-algebra~\(A\) and hence, by the Skolem-Noether theorem, there is
an element \(\alpha\) in \(A^{\times}\) such that \(J^2 = \Int(\alpha^{-1})\). 
We have for all \(x\in A\) 
\[ \alpha^{-1} J(x) \alpha =J^2(J(x))= 
J^3(x)= J(J^2(x))=
 J(\alpha^{-1} x \alpha) =J(\alpha ) J(x) J(\alpha )^{-1},
\] 
so 
\(t=\alpha J(\alpha )\in \rZ_F(A)=F\).  If \(\alpha\in F\) we have finished: \(J\) is an involution.
Now, if \(\alpha\notin F\), since \(t=a^2\) with \(a\in F^\times\), we replace \(\alpha\) with \(a^{-1}\alpha\) and  get \(\alpha J(\alpha )=1\). We
 seek \(\beta\in A^\times\) such that \(\Int(\beta)\circ J\) is an involution:  
\[(\Int(\beta)\circ J)^2(x)=
\beta J \big(\beta J(x)\beta^{-1} \big)\beta^{-1}=
\beta J(\beta^{-1}) J^2(x) J(\beta)\beta^{-1}=
\beta J(\beta^{-1}) \alpha^{-1} x \alpha J(\beta)\beta^{-1}=\gamma^{-1}x\gamma. 
\]
So we need that \(\gamma=\alpha J(\beta)\beta^{-1}\in F\). Let \(\beta=1+\alpha\): \(\beta\) is nonzero since \(\alpha\notin F\). If \(\beta\) is noninvertible, \(A=\MM_q(B)\) for some~\(B\) and \(q>1\) and we are done by induction.  If \(\beta\) is invertible, we are done  since \(\alpha J(1+ \alpha)=\alpha(1+J(\alpha))=\alpha+1=\beta\).\\
 If the involution \(J'\) we have found is symplectic we have finished. If \(J'\)
is orthogonal, since \([A_{J'}^-:F]=N(N-1)/2\geq 3\), there exists a nonzero \(y\in A\) such that \(y+J'(y)=0\). If \(y\) is noninvertible, we are done by induction. If \(y\) is invertible, \(I=\Int(y)\circ J'\) is a symplectic involution. In fact, \(I\) is clearly an involution and 
\[
a=I(a)\iff ay=yJ'(a)\iff ay=-J'(y)J'(a) \iff ay=-J'(ay) \iff ay\in A_{J'}^-,
\] 
so that \([A_{J'}^-:F]=[A_I^+:F]=N(N-1)/2\).
\end{proof}

\begin{lemma}\label{lem4-}
Let \(N>2\) and assume that every monic polynomial in \(F[X]\) of degree \(< N\) has a root in~\(F\).
Then any central simple \(F\)-algebra~\(B\) of degree \(<N\) is split.
\end{lemma}
\begin{proof} \cha{Same proof as in Lemma \ref{algclos}}.
%If \([B:F]=1\), \(B\) is split.
%If \(\deg_F(B)>1\), let \(x\in B\setminus F\). We have \(P(x)=0\) for \(P=\Cprd(x)\) with \(\deg(P)<N\).
%Rewrite \(P\) as \(\prod_j(X-\xi_j)\) (with \(\xi_j\in F\)) so that \(\prod_j(x-\xi_j)=0\). 
%At least one \(x-\xi_j\) is noninvertible.   So we find a nonzero noninvertible element in~\(B\). Hence \(B\simeq \MM_r(B_1)\) with \(r>1\). We are done by induction on \(\deg_F(B)\).
\end{proof}
%

%:lem4
\begin{lemma}\label{lem4}
Let \(N>2\) and assume that every polynomial of degree \(\leq \sup(2,N/2)\) has a root in~\(F\).
Let~\(K\) be an extension of degree~\(N\) of~\(F\).
Let \(A\in \Br_2(F)\) such that \(A\) is split by~\(K\).
Then \(A\) is already split by~\(F\).
\end{lemma}
 
\begin{proof} %Our proof is by induction on~\(N\).
 We can find an algebra \(B\sim_F A\) which contains~\(K\) as a maximal subfield  with \(\deg_F(B)=N\) by Lemma~\ref{lemBr4}. We have to show that \(B\) is split by~\(F\). % So \(B\) replaces~\(A\). 
 Using Lemma \ref{lem41}, either \(B\) is split, or \(N\) is even and we find a symplectic involution~\(J\) on~\(B\). We have \([B_J^+:F]=N(N-1)/2\geq 6\), 
and any \(x\in B_J^+\setminus F\) satisfies \(\Cpf(x)(x)=0\) with \(\deg(\Cpf(x))=N/2\). As in Lemma \ref{algclos} we find a nonzero noninvertible element in~\(B\). So \(B=\MM_r(B_1)\) with \(r>1\), and Lemma \ref{lem4-} applies for \(B_1\) since \(\deg_F(B_1)\leq  N/2<N\). %If \(x\) is noninvertible, then \(B\simeq \MM_r(B_1)\) with \(r>1\) and we are done by induction.
%By hypothesis  \(P:=\Cpf(x)\) has a root \(\xi\) in~\(F\).
%So \((x-\xi)R(x)=0\) with \(R=P/(X-\xi)\in F[X]\) of degree \(N/2-1\) with \(x-\xi\neq 0\).
%So either \(x-\xi\) is noninvertible  or \(R(x)=0\). In the last case we continue: \(R\) has a zero \(\xi'\in F\) and so on.   So we find a nonzero noninvertible element in~\(B\). Hence \(B\simeq \MM_r(B_1)\) with \(r>1\).
%If \(B_1=F\) we are done. In the other case let \(x\in B_1\setminus F\). 
%We have \(\Cprd(x)(x)=0\) with \(\deg(\Cprd(x))=N/r\leq N/2\). So \(Q=\Cprd(x)\)
%has a root in~\(F\), which leads to a nonzero noninvertible element in \(B_1\)
%and \(B_1\simeq \MM_s(B_2)\) with \(s>1\).   
%We are done by induction.
% and we are done by induction. \ttt{ je ne saisis pas ce dernier argument, il me semble qu'il faudrait savoir que \(A\) est split si \([A:F]< N\).}
\end{proof}
 
%% Now we give an algorithmic consequence of Lemma \ref{lem4} in line with Corollaries \ref{remcor1} and \ref{remprop}. % similarly to Corollary \ref{remcor1} with respect to Corollary~\ref{cor1}.

%% %c
%% %:     Corollary{remlem4}
%% \begin{corollary} \label{remlem4}
%% Let~\(K\) be an extension of degree~\(N\) of~\(F\).
%% Let \(A\in \Br_2(F)\) such that \(A\) is split by~\(K\).
%% Then \(A\) is already split by~\(F\) if we add to~\(F\) successively a finite number of roots of  polynomials of degree \(\leq \sup(2,N/2)\).\footnote{This means that we replace~\(F\) by \(F[y_1,\dots,y_\ell]\) with each \(y_j\) of degree \(\leq \sup(2,N/2)\) over \(F[(y_i)_{i<j}]\).} 
%% \end{corollary}
%--------- fin corollary ------------------------------- 

\subsection{Constructive versions of Becher's Theorem.}

%Our first constructive version is with the hypothesis made for section \ref{secBecher}: 
%the discrete field~\(F\) is supposed to be fully factorial.

%\hum{«Fully factorial» à éliminer en incorporant l'argument du théorème~\ref{th30} dans le théorème ci-dessous.}

%\ttt{En fait, on peut supposer seulement que \(F\) est contenu dans un corps algé\-briquement clos. Le résultat final est alors que l'on ne saura pas quelles extensions sont vraiment de degré~\(2\), elles seront seulement toutes de degré au plus~\(2\), pas forcément stric\-tement finies.
%Par ailleurs, tous les énoncés de la section 5 ont été donnés dans le cas fully factorial, qui donne des extensions strictement finies, mais on aurait pu seulement supposer, dans toute la section, que \(F\) est contenu dans un corps algébriquement clos.}

We first combine Lemmas \ref{lem1}, \ref{lem2} and \ref{lem4} to get the following.

\begin{lemma}\label{lem5}
  Let \(F\) be a discrete field of characteristic \(\neq 2\) and \(K = F[y,x_1,\dots,x_n]\) a triangular field extension, where $y$ is of degree \(N>2\) and
  \(x_i\) of degree \(2\). Assume that every polynomial of degree \(<N\) has a root in \(F\), then an element of \(\Br_2(F)\) split by \(K\) is already split by \(F\).
\end{lemma}

We can now prove a constructive version of Becher's Theorem \ref{thBecher}.

\begin{theorem} \label{thBecher2}
Let \(F\) be a discrete field of characteristic \(\neq 2\). Every element of \(\Br_2(F)\) is split by a commutative \(F\)-algebra~\(K\) obtained by a tower of quadratic algebra extensions. 
\end{theorem}

\begin{proof}
We can combine the results of the three previous subsections to reduce a splitting sequence to a quadratic splitting sequence.
 
The argument will be the following: we start with a commutative \(F\)-algebra \(K=F[x_1,\dots,x_\ell]\) which splits~\(A\), given explicitly by Theorem~\ref{th3}. Each \(x_i\) is a root of a monic polynomial of degree \(n_i\) over \(F[(x_j)_{j<i}]\). We say that \((n_1, \dots, n_\ell)\)
 is a splitting sequence for \(A_F\). Our aim is to replace the sequence \((n_1,\dots,n_\ell)\) by a finite sequence of~\(2\)'s. 

If we have \(n_j = 2\) for all~\(j\), 
then there is nothing to do. Otherwise
we write the sequence on the form \((\sigma, N, 2, \dots, 2)\) with \(N>2\)
and we show that we can find another splitting sequence \((\sigma, m_1,\dots,m_r)\) where all \(m_j\) are \(<N\). This means that the new sequence has strictly decreased for an order type \(\omega^\omega\) given by \((p_1,\dots,p_r)>(q_1,\dots,q_s)\) if the multiset defined by \((p_1,\dots,p_r)\) is strictly greater than the multiset defined by \((q_1,\dots,q_s)\) for the lexicographic order.\footnote{E.g.\ the multiset \([5,5,5,4,3,3,2,2]\) is greater than the multisets \([5,5,4,4,4,4,3,3,3,3,3,3,3,2,2,2,2,2,2,2]\) and \([5,5,5,4,3,2,2,2,2,2,2,2,2]\).} So our algorithm stops in a finite number of steps.

Now we explain how we manage to replace \((\sigma, N, 2, \dots, 2)\) by  
\((\sigma, m_1, \dots, m_r)\). 
%We consider the extension
% \(K[x]=F[x_1,\dots,x_k,x _{k+1}]\),  \(x = x _{k+1}\) root of a  monic polynomial  of degree~\(N\).

If some monic polynomial in this sequence has some proper factor, we can reduce the sequence. So we can assume that each polynomial is irreducible, i.e.\ the
extension \(K\) is a field extension.

\smallskip 
Let \(K_1\) be the commutative \(F\)-algebra constructed with the sequence \((\sigma)\) and \(K_2=K_1[y]\) the commutative \(F\)-algebra constructed with the sequence \((\sigma,N)\). The commutative \(F\)-algebra \(K_2\) splits~\(A\) after a sequence of quadratic extensions. In the following argument, we may assume dynamically that~\(K_1\) and~\(K_2\) are fields.

% In order that \(K_2\) splits~\(A\), it is sufficient, by Corollary \ref{remcor1}, that it splits certain (finitely many) quaternion algebras. 
 
Using Lemma \ref{lem5}, we can extend \(K_1\) by adding some roots of polynomials of degree~\(<N\) in order that the corresponding extension splits~\(A\). 
\end{proof}

% In the second constructive version we don't assume that \(F\) is fully factorial. The conclusion uses now a tower of quadratic algebra extensions.

% \begin{theorem} \label{thBecher2}
% Let \(F\) be a discrete field of characteristic \(\neq 2\). Every element of \(\Br_2(F)\) is split by a commutative \(F\)-algebra~\(K\) obtained by a tower of quadratic algebra extensions. 
% \end{theorem}
% %
% \begin{proof}
% This follows from the proof of Theorem \ref{thBecher1}, by using the dynamical way of transforming the algorithms.
% \end{proof}
% %

Note that if, after the end of our algorithm/construction, we find a nonzero noninvertible element in~\(K\), we are able to simplify the tower by cancelling one or several (useless) extensions \(K_{i+1}\simeq K_i\times K_i\).

%The improvement of Theorem \ref{thBecher2} with respect to Theorem \ref{thBecher1}
%is that there is no need to factorise polynomials in order to get the result, and the new version is nevertheless directly equivalent to Theorem \ref{thBecher} in classical mathematics. 

\smallskip Theorem \ref{thBecher2} is directly equivalent to Theorem \ref{thBecher} in classical mathematics.

%It does not seem possible to use the present argument to prove that we can split an element of \(\Br_2(F)\) by {\em independent} quadratic extensions
%of \(F\), which is also a consequence of Merkurjev's Theorem. 

\section{Fully factorial fields}

A discrete field~\(K\) is said to be \emph{fully factorial} when any polynomial in \(L[X]\) with \(L\)  any strictly finite field extension of~\(K\) can be factorised into irreducible factors \citep[see][Chapter VII]{MRR}.
%If \(K\) is a subfield of an algebraically closed field~\(L\), with \(K\) detachable in~\(L\), then \(K\) is fully factorial.
In the case where \(F\) is fully factorial, we %can refine 
\cha{obtain the following variation of} Theorem \ref{thBecher2}.

\begin{theorem} \label{thBecher3}
  Let \(F\) be a fully factorial discrete field of characteristic \(\neq 2\).
  Every element of \(\Br_2(F)\) is split by a finite extension of~\(F\) obtained by a tower of quadratic extensions. 
\end{theorem}

%% %:     Theorem{th30}
%% \begin{theorem} \label{th30}
%% \cha{Let \(F\) be a fully factorial field and \(A\)  a central simple
%%   \(F\)-algebra.
%% We can construct a field extension~\(K\) of~\(F\) that is strictly finite and splits~\(A\).}
%%  \end{theorem}
%----------- fin theorem ----------------------------- 

\bibliographystyle{myauthordate}
\bibliography{central}

\end{document}